\documentclass[12pt, a4paper, twoside]{amsart}
\usepackage{amsmath}
\usepackage{amssymb}

\usepackage{amsfonts}
\usepackage{latexsym, amsthm, mathrsfs}
\usepackage[german, english]{babel}

\newcommand{\al}{\alpha}
\newcommand{\om}{\omega}

\newcommand{\vp}{\varphi}

\newcommand{\sse}{\subseteq}
\newcommand{\contains}{\supseteq}

\DeclareMathOperator{\depth}{depth}

\DeclareMathOperator{\dom}{dom}

\DeclareMathOperator{\E}{E}

\DeclareMathOperator{\R}{R}
\DeclareMathOperator{\Ext}{Ext}

\newcommand{\rgl}{\rangle}
\newcommand{\lgl}{\langle}

\newcommand{\re}{\restriction}

\newcommand{\stem}{\mathrm{stem}}

\newcommand{\bT}{\mathbb{T}}

\newcommand{\bP}{\mathbb{P}}

\newcommand{\bN}{\mathbb{N}}

\newcommand{\bS}{\mathbb{S}}

\newcommand{\ra}{\rightarrow}

\newcommand{\Erdos}{Erd\H{o}s}

\newcommand{\Rodl}{R{\"{o}}dl}

\begin{document}

\newtheorem{thm}{Theorem}  
\newtheorem{prop}[thm]{Proposition} 
\newtheorem{lem}[thm]{Lemma} 
\newtheorem{cor}[thm]{Corollary} 
\newtheorem{fact}[thm]{Fact}
\newtheorem{claim}[thm]{Claim}
\newtheorem*{thmMT}{Main Theorem}
\newtheorem*{thmnonumber}{Theorem}
\newtheorem*{mainclaim}{Main Claim}
\newtheorem*{claim1}{Claim 1}
\newtheorem*{claim2}{Claim 2}
\newtheorem*{claim3}{Claim 3}
\newtheorem*{claim4}{Claim 4}
\newtheorem*{claim5}{Claim 5}
\newtheorem*{claim6}{Claim 6}
\newtheorem*{claim7}{Claim 7}
\newtheorem*{claim8}{Claim 8}
\newtheorem*{claim9}{Claim 9}
\newtheorem*{claim10}{Claim 10}
\newtheorem*{claim11}{Claim 11}
\newtheorem*{claim12}{Claim 12}
\newtheorem*{claim13}{Claim 13}
\newtheorem*{claim14}{Claim 14}
\newtheorem*{claim15}{Claim 15}
\newtheorem*{claimA}{Claim A}
\newtheorem*{claimC}{Claim C}
\newtheorem*{claimD}{Claim D}
\newtheorem*{factD}{Fact D}
\newtheorem{claimn}{Claim}

\theoremstyle{definition}   
\newtheorem{defn}[thm]{Definition} 
\newtheorem{example}[thm]{Example} 
\newtheorem{conj}[thm]{Conjecture} 
\newtheorem{prob}[thm]{Problem} 
\newtheorem{examples}[thm]{Examples}
\newtheorem{question}[thm]{Question}
\newtheorem{problem}[thm]{Problem}
\newtheorem{openproblems}[thm]{Open Problems}
\newtheorem{openproblem}[thm]{Open Problem}
\newtheorem{conjecture}[thm]{Conjecture}
\newtheorem*{problem1}{Problem 1}
\newtheorem*{problem2}{Problem 2}
\newtheorem*{notn}{Notation}

\theoremstyle{remark} 
\newtheorem*{rem}{Remark} 
\newtheorem*{rems}{Remarks} 
\newtheorem*{ack}{Acknowledgments} 
\newtheorem*{note}{Note}
\newtheorem{subclaim}{Subclaim}
\newtheorem*{subclaimn}{Subclaim}
\newtheorem*{subclaim1}{Subclaim (i)}
\newtheorem*{subclaim2}{Subclaim (ii)}
\newtheorem*{subclaim3}{Subclaim (iii)}
\newtheorem*{subclaim4}{Subclaim (iv)}
\newtheorem*{case1}{Case 1}
\newtheorem*{case2}{Case 2}
\newtheorem*{case3}{Case 3}
\newtheorem*{case4}{Case 4}
\newtheorem*{case5}{Case 5}
\newtheorem*{case6}{Case 6}
\newtheorem*{case7}{Case 7}
\newtheorem*{case8}{Case 8}
\newtheorem*{case9}{Case 9}
\newtheorem*{case10}{Case 10}
\newtheorem*{case11}{Case 11}
\newtheorem*{case12}{Case 12}
\newtheorem*{case13}{Case 13}
\newtheorem*{case14}{Case 14}
\newtheorem*{case15}{Case 15}
\newtheorem*{subcasei}{Subcase (i)}
\newtheorem*{subcaseii}{Subcase (ii)}
\newtheorem*{subcaseiii}{Subcase (iii)}
\newtheorem*{subcaseiv}{Subcase (iv)}
\newtheorem*{subcasev}{Subcase (v)}
\newtheorem*{subcasevi}{Subcase (vi)}
\newtheorem*{subsubcasea}{Sub-subcase (a)}
\newtheorem*{subsubcaseb}{Sub-subcase (b)}

\numberwithin{equation}{section}

\title[Ramsey-Classification Theorems]{A new class of Ramsey-Classification Theorems 	and their Applications in the Tukey Theory of  
Ultrafilters}

\author{Natasha Dobrinen} 
\address{Department of Mathematics, University of Denver, 2360 S Gaylord St, Denver, CO 80208, USA} 
\email{natasha.dobrinen@du.edu} 
\thanks{The first author was supported by a National Science Foundation - Association for Women in Mathematics Mentoring Grant and a University of Denver Faculty Research Fund Grant}
  
\author{Stevo Todorcevic}
\address{Department of Mathematics\\
University of Toronto\\
Toronto\\ Canada\\ M5S 2E4}
\email{stevo@math.toronto.edu}
\address{Institut de Mathematiques de Jussieu\\
CNRS - UMR 7056\\
75205 Paris\\
France}
\email{stevo@math.jussieu.fr}
\thanks{The second author was supported by grants from NSERC and CNRS}


\begin{abstract}
Motivated by  Tukey classification problems and building on work in \cite{Dobrinen/Todorcevic11},
we develop a  new hierarchy of 
topological Ramsey spaces $\mathcal{R}_{\al}$, $\al<\om_1$.
These spaces form a natural hierarchy of complexity, 
$\mathcal{R}_0$ being the Ellentuck space \cite{MR0349393},
and for each $\al<\om_1$, $\mathcal{R}_{\al+1}$ coming immediately after $\mathcal{R}_{\al}$ in complexity. 
Associated with each $\mathcal{R}_{\al}$ is an
 ultrafilter $\mathcal{U}_{\al}$, which is Ramsey for $\mathcal{R}_{\al}$, and in particular, 
 is a  rapid p-point satisfying certain partition properties.
We prove Ramsey-classification theorems for equivalence relations on
fronts on $\mathcal{R}_{\al}$, $2\le\al<\om_1$. 
These are analogous to the  Pudlak-\Rodl\
Theorem canonizing  equivalence relations on barriers on the
Ellentuck space. 
We then apply our Ramsey-classification theorems to
completely classify all Rudin-Keisler equivalence classes of
ultrafilters which are Tukey reducible to $\mathcal{U}_{\al}$, for each $2\le\al<\om_1$: Every
ultrafilter which is Tukey reducible to $\mathcal{U}_{\al}$ is
isomorphic to a countable iteration of Fubini products of
ultrafilters from among a fixed countable collection of rapid p-points.
Moreover, we show that the Tukey types of nonprincipal ultrafilters Tukey reducible to $\mathcal{U}_{\al}$ form a descending chain of order type $\al+1$.
\end{abstract}

\maketitle

\section{Overview}\label{sec.intro}

This paper builds on and extends work in \cite{Dobrinen/Todorcevic11} to a large class of new topological Ramsey spaces and their associated ultrafilters.
Motivated by a Tukey classification problem and inspired by work of Laflamme in \cite{Laflamme89} and the second author in \cite{Raghavan/Todorcevic11},
we build  new topological Ramsey spaces $\mathcal{R}_{\al}$, $2\le\al<\om_1$.
The space $\mathcal{R}_0$ denotes the classical Ellentuck space;
the space $\mathcal{R}_1$ was built in \cite{Dobrinen/Todorcevic11}.
The 
topological Ramsey spaces $\mathcal{R}_{\al}$, $\al<\om_1$,  form a natural hierarchy in terms of complexity.
The space $\mathcal{R}_1$ is minimal in complexity above the  Ellentuck space,
the Ellentuck space being obtained as the projection of $\mathcal{R}_1$ via
  a fixed finite-to-one map.
More generally, $\mathcal{R}_{\al+1}$ is minimal in complexity over $\mathcal{R}_{\al}$ via a fixed finite-to-one map.
For limit ordinals $\gamma<\al$, $\mathcal{R}_{\gamma}$ is formed by diagonalizing in a precise manner over the $\mathcal{R}_{\beta}$, $\beta<\gamma$.
$\mathcal{R}_{\gamma}$ is minimal in complexity over the collection of $\mathcal{R}_{\beta}$, $\beta<\gamma$, via fixed finite-to-one maps.

Every topological Ramsey space has  notions of finite approximations, fronts, and barriers.
In \cite{Dobrinen/Todorcevic11}, we proved that
for each $n$,  there is a finite collection of canonical equivalence relations for uniform barriers on $\mathcal{R}_1$ of rank $n$.
In this paper, we prove similar results for all $\al<\om_1$.
In Theorem \ref{thm.FCTR^al_n},
 we  show that for all $2\le\al<\om_1$,
 for any uniform barrier $\mathcal{B}$ on $\mathcal{R}_{\al}$ of finite rank  and 
any equivalence relation $\E$ on $\mathcal{B}$,
there is an $X\in\mathcal{R}_{\al}$ such that $\E$ restricted to the members of  $\mathcal{B}$ coming from within $X$ is exactly one of the canonical equivalence relations.
For finite $\al$, there are finitely many
 canonical equivalence relations on uniform barriers of finite rank; these are represented
by a certain collection of finite trees.
Moreover, the numbers of canonical equivalence relations for finite $\al$ are given by a recursive function.
For infinite $\al$, there are infinitely many canonical equivalence relations on uniform barriers of finite rank, represented by tree-like structures.
These theorems  generalize the \Erdos-Rado Theorem  for uniform barriers of finite rank on the Ellentuck space, namely, those of the form $[\bN]^n$.

In the  main theorem of this paper,
Theorem \ref{thm.PRR_al},
we prove  new Ramsey-classification theorems for all barriers (and moreover all fronts) on the topological Ramsey spaces $\mathcal{R}_{\al}$, $2\le\al<\om_1$.
We prove that for any barrier $\mathcal{B}$ on $\mathcal{R}_{\al}$ and any equivalence relation  on $\mathcal{B}$,
there is an  inner Sperner map which canonizes the equivalence relation.
This generalizes our analogous theorem for $\mathcal{R}_1$ in \cite{Dobrinen/Todorcevic11}, which in turn generalized the Pudlak-\Rodl\ Theorem for  barriers on the Ellentuck space.
These classification theorems were motivated by the following.

Recently the second author (see Theorem 24 in
\cite{Raghavan/Todorcevic11}) has made a connection between the
Ramsey-classification theory (also known as the canonical Ramsey
theory) and the Tukey classification theory of ultrafilters  on
$\omega.$ 
More precisely, he showed that selective ultrafilters
realize  minimal Tukey types in the class of all ultrafilters on
$\omega$ by applying the Pudlak-R\"{o}dl Ramsey classification
result to a given cofinal map from a selective ultrafilter into
any other ultrafilter on $\omega,$ a map which, on the basis of
our previous paper \cite{Dobrinen/Todorcevic10}, he could assume
to be continuous. Recall that the notion of a selective
ultrafilter is closely tied to the Ellentuck space on the family of
all infinite subsets of $\omega,$ or rather the one-dimensional
version of the pigeon-hole principle on which the Ellentuck space is
based, the principle stating that an arbitrary
$f:\omega\rightarrow\omega$ is either constant or is one-to-one on
an infinite subset of $\omega.$ Thus an ultrafilter $\mathcal{U}$
on $\omega$ is \emph{selective} if for every map
$f:\omega\rightarrow\omega$ there is an $X\in \mathcal{U}$ such that
$f$ is either constant or one-to-one on $\mathcal{U}.$ Since
essentially any other topological Ramsey space has it own notion
of a selective ultrafilter living on the set of its
$1$-approximations (see \cite{MR2330595}), the argument for
Theorem 24 in \cite{Raghavan/Todorcevic11} is so general that it
will give analogous Tukey-classification results for all
ultrafilters of this sort provided, of course, that we have the
analogues of the Pudlak-R\"{o}dl Ramsey-classification result for
the corresponding topological Ramsey spaces.

In \cite{Laflamme89},
Laflamme forced ultrafilters, $\mathcal{U}_{\al}$, $1\le \al<\om_1$, which are rapid p-points satisfying certain partition properties, and which have complete
combinatorics over the Solovay model. 
Laflamme showed that, for each $1\le\al<\om_1$, the Rudin-Keisler equivalence classes of all ultrafilters Rudin-Keisler below $\mathcal{U}_{\al}$ form a descending chain of order type $\al+1$.
This result employs a result of Blass in \cite{Blass74} which states that each weakly Ramsey ultrafilter has exactly one Rudin-Keisler type below it, namely the isomorphism class of a selective ultrafilter.
At this
point it is instructive to recall another result of the second
author (see Theorem 4.4 in \cite{MR1644345}) stating that assuming
sufficiently strong large cardinal axioms \emph{every} selective
ultrafilter is generic over $L(\mathbb{R})$ for the partial order
of infinite subsets of $\omega$, and the same argument applies for
any other ultrafilter that is selective relative any other
topological Ramsey space (see \cite{MR2330595}). 
Since, as it is
well-known, assuming large cardinals, the theory of
$L(\mathbb{R})$ cannot be changed by forcing, this gives another
perspective to the notion of `complete combinatorics' of Blass
and Laflamme.

This  paper was motivated by the same two lines of motivation as in \cite{Dobrinen/Todorcevic11}.
One line of motivation was to find  the structure of the Tukey types of nonprincipal ultrafilters Tukey reducible to $\mathcal{U}_{\al}$, for all $1\le\al<\om_1$.
We show in Theorem \ref{cor.1Tpred}
that, in fact, the   Tukey types of nonprincipal  ultrafilters   below that of $\mathcal{U}_{\al}$ forms a descending chain of order type $\al+1$.
Thus, the structure of the Tukey types below $\mathcal{U}_{\al}$ is the same as the structure of the Rudin-Keisler equivalence classes below $\mathcal{U}_{\al}$.

The second and stronger motivation  was to find  new canonization theorems for equivalence relations on  fronts on  $\mathcal{R}_{\al}$, $2\le\al<\om_1$, and to apply them to obtain 
 finer results than  Theorem \ref{cor.1Tpred}.
The canonization Theorems \ref{thm.original}  and \ref{thm.PRR_al} generalize results in \cite{Dobrinen/Todorcevic11}, which in turn had generalized
the \Erdos-Rado Theorem  for barriers of the form $[\bN]^n$ and the Pudlak-\Rodl\ Theorem for general barriers on the Ellentuck space, respectively.

Each of the spaces $\mathcal{R}_{\al}$ is constructed to give rise to an  ultrafilter which is isomorphic to Laflamme's $\mathcal{U}_{\al}$.
Applying Theorem \ref{thm.PRR_al},
we  completely classify all Rudin-Keisler classes of ultrafilters which are 
 Tukey reducible to  $\mathcal{U}_{\al}$ in Theorem \ref{thm.TukeyU_al}.
These extend the authors' Theorem 5.10 in \cite{Dobrinen/Todorcevic11}, which itself
 extended the second author's
Theorem 24 in
\cite{Raghavan/Todorcevic11}, classifying the Rudin-Keisler classes within  the Tukey type of a Ramsey ultrafilter.

The main new contributions in this work, as opposed to straightforward generalizations of the work in \cite{Dobrinen/Todorcevic11}, are the following.
First, the cases when $\al$ is infinite necessitate a new way of constructing the spaces $\mathcal{R}_{\al}$.
The base trees $\bT_{\al}$ for the spaces $\mathcal{R}_{\al}$ must be well-founded in order to generate topological Ramsey spaces.
However,  the true structures are best captured by tree-like objects $\bS_{\al}$ which are neither truly trees nor  well-founded.
These new auxiliary structures $\bS_{\al}$  are also needed to  make the canonical equivalence relations precise.
Second, we provide a general induction scheme by which we prove the Ramsey-classification theorems hold for $\mathcal{R}_{\al}$, for all $\al<\om_1$.
The proof that $\mathcal{R}_{\al}$ is a topological Ramsey space uses the Ramsey-classification theorems for all $\mathcal{R}_{\beta}$, $\beta<\al$.
Third, new sorts of structures appear within the collection of all Rudin-Keisler equivalence classes lying within the Tukey type of $\mathcal{U}_{\al}$, for $\al\ge 2$. 
Taken together, these constitute the first known transfinite collection of topological Ramsey spaces with associated ultrafilters which, though very far from Ramsey, behave quite similarly to Ramsey ultrafilters.
We remark that the fact that each $\mathcal{R}_{\al}$ is a topological Ramsey space is essential to the proof of Theorem \ref{thm.TukeyU_al}, and that forcing alone is not sufficient to obtain our result.

\section{Introduction, Background and Definitions}\label{sec.bd}

We begin with some definitions and background 
for the results in this paper.
Let $\mathcal{U}$ be an ultrafilter on a countable base set.
A subset $\mathcal{B}$ of an ultrafilter $\mathcal{U}$ is called {\em cofinal} if it is a base for the ultrafilter $\mathcal{U}$; that is, if for each $U\in\mathcal{U}$ there is an $X\in\mathcal{B}$ such that $X\sse U$.
Given ultrafilters $\mathcal{U},\mathcal{V}$,
we say that a function $g:\mathcal{U}\ra \mathcal{V}$
 is {\em cofinal} if the image of each cofinal subset of $\mathcal{U}$ is cofinal in $\mathcal{V}$.
We say that $\mathcal{V}$ is {\em Tukey reducible} to $\mathcal{U}$, and write $\mathcal{V}\le_T \mathcal{U}$, if there is a cofinal map from $\mathcal{U}$ into $\mathcal{V}$.
If both $\mathcal{V}\le_T \mathcal{U}$ and $\mathcal{U}\le_T \mathcal{V}$, then we write $\mathcal{U}\equiv_T \mathcal{V}$ and say that $\mathcal{U}$ and $\mathcal{V}$ are Tukey equivalent.
$\equiv_T$ is an equivalence relation, and  $\le_T$ on the equivalence classes forms a partial ordering.
The equivalence classes are called {\em Tukey types}.
A cofinal map $g:\mathcal{U}\ra\mathcal{V}$ is called {\em monotone} if whenever $U\contains U'$ are elements of $\mathcal{U}$, we have $g(U)\contains g(U')$.
By Fact 6 in \cite{Dobrinen/Todorcevic10},
  $\mathcal{U}\ge_T\mathcal{V}$ if and only if there is a monotone cofinal map witnessing this.
It is  useful to note that  $\mathcal{U}\ge_T\mathcal{V}$ if and only if there are cofinal subsets $\mathcal{B}\sse\mathcal{U}$ and $\mathcal{C}\sse\mathcal{V}$ and a  map $g:\mathcal{B}\ra\mathcal{C}$ which is a cofinal map from $\mathcal{B}$ into $\mathcal{C}$.

Rudin-Keisler reducibility is defined as follows.
 $\mathcal{U}\le_{RK}\mathcal{V}$ if and only if there is a function $f:\om\ra\om$ such that $\mathcal{U}=f(\mathcal{V})$, where
\begin{equation}
f(\mathcal{V})=\lgl\{f(U):U\in\mathcal{U}\}\rgl.
\end{equation}
Recall that $\mathcal{U}\equiv_{RK}\mathcal{V}$ if and only if $\mathcal{U}$ and $\mathcal{V}$ are isomorphic.
Tukey reducibility on ultrafilters generalizes  Rudin-Keisler reducibility in that $\mathcal{U}\ge_{RK} \mathcal{V}$ implies that $\mathcal{U}\ge_T\mathcal{V}$.
The converse does not hold.
In particular, there are $2^{\mathfrak{c}}$ many ultrafilters in the top Tukey type, \cite{Juhasz66},  \cite{Isbell65}. 
(See \cite{Dobrinen/Todorcevic10}, \cite{Raghavan/Todorcevic11}, and \cite{Dobrinen/Todorcevic11} for more examples of Tukey types containing more than one Rudin-Keisler equivalence class.)

We remind the reader of the following special kinds of ultrafilters.

\begin{defn}[\cite{Bartoszynski/JudahBK}]
Let $\mathcal{U}$ be an ultrafilter on $\om$.
\begin{enumerate}
\item
$\mathcal{U}$ is {\em Ramsey} if for each coloring $c:[\om]^2\ra 2$, there is a $U\in\mathcal{U}$ such that $U$ is homogeneous, meaning $|c''[U]^2|=1$.
\item
$\mathcal{U}$ is {\em weakly Ramsey} if for each coloring $c:[\om]^2\ra 3$, there is a $U\in\mathcal{U}$ such that $|c''[U]^2|\le 2$.
\item
$\mathcal{U}$ is a {\em p-point} if for each decreasing sequence $U_0\contains U_1\contains\dots$ of elements of $\mathcal{U}$, there is an $X\in\mathcal{U}$ such that $|X\setminus U_n|<\om$, for each $n<\om$.
\item
$\mathcal{U}$ is  {\em rapid} if for each function $f:\om\ra\om$, there is an $X\in\mathcal{U}$ such that $|X\cap f(n)|\le n$ for each $n<\om$.
\end{enumerate}
\end{defn}

Every Ramsey ultrafilter is weakly Ramsey,  which is in turn both a  p-point and rapid.
These sorts of ultrafilters  exist in every model of CH or MA or under some weaker cardinal invariant assumptions (see \cite{Bartoszynski/JudahBK}).
Ramsey ultrafilters are also called {\em selective}, and the property of being Ramsey is equivalent to the following property:
For each decreasing sequence $U_0\contains U_1\contains\dots$ of members of $\mathcal{U}$, there is an $X\in\mathcal{U}$ such that  for each $n<\om$, $X\sse^* U_n$ and moreover $|X\cap(U_{n+1}\setminus U_n)|\le 1$.

Any subset of $\mathcal{P}(\om)$ is a   topological space, with the subspace topology inherited from the Cantor space.
Thus, given any
$\mathcal{B},\mathcal{C}\sse\mathcal{P}(\om)$,
 a function $g:\mathcal{B}\ra\mathcal{C}$ is continuous if for each sequence $(X_n)_{n<\om}\sse\mathcal{B}$ which converges to some $X\in\mathcal{B}$,
the sequence $(g(X_n))_{n<\om}$ converges to $g(X)$, meaning that  for all $k$ there is an $n_k$ such that for all $n\ge n_k$,
$g(X_n)\cap k=g(X)\cap k$.
For any ultrafilter $\mathcal{V}$, cofinal $\mathcal{C}\sse\mathcal{V}$, and  $X\in\mathcal{V}$,  we use $\mathcal{C}\re X$ to denote $\{Y\in\mathcal{C}:Y\sse X\}$.
 $\mathcal{C}\re X$ is a cofinal subset of $\mathcal{V}$ and hence is a filter base for $\mathcal{V}$.
Thus, $(\mathcal{V},\contains)\equiv_T(\mathcal{C}\re X,\contains)$.

The authors  proved in Theorem 20 of \cite{Dobrinen/Todorcevic10} that if $\mathcal{U}$ is a p-point and $\mathcal{W}\le_T\mathcal{U}$,
then there is a continuous monotone cofinal map witnessing this.

\begin{thm}[Dobrinen-Todorcevic \cite{Dobrinen/Todorcevic10}]\label{thm.5}
Suppose $\mathcal{U}$ is a p-point on $\bN$ and that $\mathcal{V}$ is an arbitrary ultrafilter on $\bN$ such that $\mathcal{V}\le_T\mathcal{U}$.
Then there is a continuous monotone map
$g:\mathcal{P}(\bN)\ra\mathcal{P}(\bN)$
 whose restriction to $\mathcal{U}$ is continuous
and has cofinal range in $\mathcal{V}$.
Hence,  $g\re\mathcal{U}$ is a continuous monotone cofinal map from $\mathcal{U}$ into $\mathcal{V}$ witnessing that $\mathcal{V}\le_T\mathcal{U}$.
\end{thm}

Tukey types of p-points has been the subject of work in \cite{Dobrinen/Todorcevic10}, \cite{Raghavan/Todorcevic11}, \cite{Dobrinen10}, and \cite{Dobrinen/Todorcevic11}, and is a sub-theme of this paper.
From Theorem \ref{thm.5},
it follows that every p-point has Tukey type of cardinality continuum. 
However, the Tukey type of a p-point is often quite different from its Rudin-Keisler isomorphism class.
In fact, it is unknown whether there is a p-point whose Tukey type coincides with its Rudin-Keisler class.
By results in \cite{Dobrinen/Todorcevic10},
such a p-point must not be rapid.

To discuss this further, the reader is reminded of the definition of the Fubini product of a collection of ultrafilters.

\begin{defn}\label{defn.Fubprod}
Let $\mathcal{U},\mathcal{V}_n$, $n<\om$, be ultrafilters.
The {\em Fubini product} of $\mathcal{U}$ and $\mathcal{V}_n$, $n<\om$,
is the ultrafilter, denoted $\lim_{n\ra\mathcal{U}}\mathcal{V}_n$,  on base set $\om\times\om$ consisting of the sets $A\sse\om\times\om$ such that
\begin{equation}
\{n\in\om:\{j\in\om:(n,j)\in A\}\in\mathcal{V}_n\}\in\mathcal{U}.
\end{equation}
That is, for $\mathcal{U}$-many $n\in\om$, the section $(A)_n$ is in $\mathcal{V}_n$.
If all $\mathcal{V}_n=\mathcal{U}$, then we let $\mathcal{U}\cdot\mathcal{U}$ denote $\lim_{n\ra\mathcal{U}}\mathcal{U}$.
\end{defn}

It is well-known that the Fubini product of two or more  p-points is not a p-point, hence for any p-point, $\mathcal{U}\cdot\mathcal{U}>_{RK}\mathcal{U}$.
The following facts stand in contrast to this.
Every Ramsey ultrafilter $\mathcal{U}$ has Tukey type equal to the Tukey type of $\mathcal{U}\cdot\mathcal{U}$, and moreover that this is the case for any rapid p-point (see Corollary 37 of \cite{Dobrinen/Todorcevic10}).
Assuming CH, there are p-points $\mathcal{U}\equiv_T\mathcal{V}$ such that $\mathcal{V}<_{RK}\mathcal{U}$ (see Theorem 25 of \cite{Raghavan/Todorcevic11}).
Assuming CH, MA, or using forcing,
the Tukey type of $\mathcal{U}_1$ (the weakly Ramsey ultrafilter constructed from the topological Ramsey space $\mathcal{R}_1$)
contains a Rudin-Keisler strictly increasing chain of order type $\om_1$; contains a Rudin-Keisler strictly increasing chain of rapid p-points of order type $\om$;
and contains ultrafilters which are Rudin-Keisler incomparable (see Example 5.17 of \cite{Dobrinen/Todorcevic10}).
Hence, although the Tukey type of any p-point has size continuum, it can contain many Rudin-Keisler inequivalent  ultrafilters within it.

The question of what precisely are the isomorphism classes within the Tukey type of a given ultrafilter has been answered for Ramsey ultrafilters and for ultrafilters  $\mathcal{U}_1$ which are Ramsey for the topological Ramsey space $\mathcal{R}_1$.
We discuss the previously known results here in order to give the context of the results of this paper.

\begin{thm}[Todorcevic,  Theorem 24, \cite{Raghavan/Todorcevic11}]\label{thm.tod}
If $\mathcal{U}$ is a Ramsey  ultrafilter and $\mathcal{V}\le_T\mathcal{U}$,
then $\mathcal{V}$ is isomorphic to a
countable iterated Fubini product of $\mathcal{U}$.
\end{thm}

The  proof of Theorem \ref{thm.tod} uses the Pudlak-\Rodl\ Theorem \ref{thm.PR}  which we  review below.
Given Theorem \ref{thm.tod}, one may reasonably ask whether a similar situation holds for ultrafilters which are not Ramsey.
The most natural place to start is low in the Rudin-Keisler hierarchy, with an ultrafilter which is weakly Ramsey but not Ramsey.
Laflamme forced such an ultrafilter, and moreover, constructed a large hierarchy of ultrafilters which are rapid p-points satisfying partition properties which give rise to complete combinatorics.

Recall from \cite{Laflamme89} that an ultrafilter $\mathcal{U}$ is said to satisfy the {\em $(n,k)$ Ramsey partition property} (or RP$^n(k)$) if for all functions $f:[\om]^k\ra n^{k-1}+1$, and all partitions $\lgl A_m:m\in\om\rgl$ of  $\om$ with each $A_m\not\in\mathcal{U}$,
there is a set $X\in\mathcal{U}$ such that $|X\cap A_m|<\om$ for each $m<\om$, and $|f''[A_m\cap X]^2|\le n^{k-1}$ for each $m<\om$.

\begin{thm}[Laflamme \cite{Laflamme89}]\label{thm.Laflammethms}
For each $1\le\al<\om_1$,
there is an ultrafilter $\mathcal{U}_{\al}$, forced by a $\sigma$-complete forcing $\bP_{\al}$, with the following properties.
\begin{enumerate}
\item
For each $1\le \al<\om$,
$\mathcal{U}_{\al}$ is a rapid p-point which has complete combinatorics.
\item
For each $1\le n<\om$,
$\mathcal{U}_n$
satisfies the $(n,k)$ Ramsey partition property for all $k\ge 1$.
For $\om\le\al<\om_1$, $\mathcal{U}_{\al}$ satisfies analogous Ramsey partition properties.
\item
The isomorphism types of all nonprincipal ultrafilters Rudin-Keisler reducible to $\mathcal{U}_{\al}$ forms strictly decreasing chain of order type $\al+1$.
\item
$\mathcal{U}_1$ is weakly Ramsey but not Ramsey.
\end{enumerate}
\end{thm}

It follows from a  theorem of Blass in \cite{Blass74} that there is only one isomorphism class of nonprincipal ultrafilters Rudin-Keisler below $\mathcal{U}_1$, which we denote $\mathcal{U}_0$.

In \cite{Dobrinen/Todorcevic11}, the authors
constructed a topological Ramsey space $\mathcal{R}_1$ which is forcing equivalent to Laflamme's forcing $\bP_1$. 
Thus, the ultrafilter associated with $\mathcal{R}_1$ is aptly named $\mathcal{U}_1$.
In  \cite{Dobrinen/Todorcevic11}, the authors
 extended Theorem \ref{thm.tod} to  $\mathcal{U}_1$ (see Theorem \ref{thm.TukeyU_1}  below), 
in the process proving a new Ramsey classification theorem for equivalence relations on fronts on the space $\mathcal{R}_1$ (see Theorem \ref{thm.PRR(1)} below).
To put this work into context, we review Ramsey's Theorem and the  canonization theorems of \Erdos-Rado and Pudlak-\Rodl\ for barriers on the Ellentuck space.
Recall that $[M]^k$ denotes the collection of all subsets of the given set $M$ with cardinality $k$.

\begin{thm}[Ramsey \cite{Ramsey29}]\label{thm.ramsey}
For every positive integer $k$ and every finite coloring of the family $[\bN]^k$,
there is an infinite subset $M$ of $\bN$ such that
the set $[M]^k$ of all $k$-element subsets of $M$ is monochromatic.
\end{thm}

When one is interested in  equivalence relations on $[\bN]^k$,
the canonical equivalence relations are determined by subsets $I\sse\{0,\dots,k-1\}$ as follows:
\begin{equation}
\{x_0,\dots,x_{k-1}\} \E_I \{y_0,\dots,y_{k-1}\}\mathrm{\ iff\ } \forall i\in I,\ x_i=y_i,
\end{equation}
where the $k$-element sets $\{x_0,\dots,x_{k-1}\}$ and $\{y_0,\dots,y_{k-1}\}$ are taken to be in increasing order.

\begin{thm}[\Erdos-Rado \cite{Erdos/Rado50}]\label{thm.ER}
For every $k\ge 1$ and every equivalence relation $\E$ on $[\bN]^k$,
there is an infinite subset $M$ of $\bN$ and an index set $I\sse\{0,\dots,k-1\}$ such that $\E\re [M]^k=\E_I\re[M]^k$.
\end{thm}

For each $k<\om$, the set $[\bN]^k$ is an example of a uniform barrier of rank $k$ for the Ellentuck space. This leads us to the more general notions of fronts and barriers.
Here, $a\sqsubset b$ denotes that $a$ is a proper initial segment of $b$.

\begin{defn}[\cite{TodorcevicBK10}]\label{defn.barrier}
Let $\mathcal{F}\sse[\bN]^{<\om}$ and $M\in[\bN]^{\om}$.
$\mathcal{F}$ is a {\em front} on $M$ if
\begin{enumerate}
\item
For each $X\in[M]^{\om}$,
there is an $a\in\mathcal{F}$ for which $a\sqsubset X$; and
\item
For all $a,b\in\mathcal{F}$ such that $a\ne b$, we have $a\not\sqsubseteq b$.
\end{enumerate}
$\mathcal{F}$ is a {\em barrier} on $M$ if (1) and (2$'$) hold,
where
\begin{enumerate}
\item[(2$'$)]
For all $a,b\in\mathcal{F}$ such that $a\ne b$, we have $a\not\sse b$.
\end{enumerate}
\end{defn}

Thus, every barrier is a front.  Moreover, by a theorem of  Galvin in \cite{Galvin68}, for every front $\mathcal{F}$, there is an infinite $M\sse\bN$ for which $\mathcal{F}|M$ is a barrier.
The Pudlak-\Rodl\ Theorem  extends the \Erdos-Rado Theorem to general barriers.
If $\mathcal{F}$ is a front, a mapping $\vp:\mathcal{F}\ra\bN$ is called {\em irreducible} if it is (a) {\em inner}, meaning that $\vp(a)\sse a$ for all $a\in\mathcal{F}$, and (b) {\em Nash-Williams},
meaning that for each $a,b\in\mathcal{F}$, $\vp(a)\not\sqsubset \vp(b)$.

\begin{thm}[Pudlak-\Rodl, \cite{Pudlak/Rodl82}]\label{thm.PR} 
For every barrier $\mathcal{F}$ on $\bN$ and every equivalence relation $\E$ on $\mathcal{F}$, there is an infinite $M\sse\bN$ such that the restriction of $\E$ to $\mathcal{F}|M$ is represented by an irreducible mapping defined on $\mathcal{F}|M$.
\end{thm}

In \cite{Dobrinen/Todorcevic11}, the authors
generalized the Pudlak-\Rodl\ Theorem to fronts on the topological Ramsey space $\mathcal{R}_1$.
To avoid unnecessary length in the introduction, we refer the reader to 
 Sections \ref{sec.stevobook}  -  \ref{sec.canonizationsRal}
for the definitions of $\mathcal{R}_1$, fronts on general topological Ramsey spaces,  and canonical equivalence relations.

\begin{thm}[Dobrinen/Todorcevic \cite{Dobrinen/Todorcevic11}]\label{thm.PRR(1)}
Suppose   $\mathcal{F}$ is  a front on $\mathcal{R}_{\al}$ and  $\R$ is an equivalence relation on $\mathcal{F}$.
Then
 there is an $A\in\mathcal{R}_{\al}$ such that $\R$ is canonical when restricted to 
$\mathcal{F}| A$.
\end{thm}

We applied Theorem \ref{thm.PRR(1)} to obtain the next result, completely classifying all isomorphism types of ultrafilters Tukey reducible to $\mathcal{U}_1$.

\begin{thm}[Dobrinen/Todorcevic \cite{Dobrinen/Todorcevic11}]\label{thm.TukeyU_1}
Suppose   $\mathcal{V}\le_T\mathcal{U}_1$.
Then $\mathcal{V}$ is isomorphic to an iterated Fubini product of ultrafilters from among a countable collection of ultrafilters.
Moreover, this countable collection 
 forms a Rudin-Keisler strictly increasing chain of order-type $\om$.
In particular, $\mathcal{U}_0$ is the Rudin-Keisler minimal  nonprincipal ultrafilter among them, and the other nonprincipal ultrafilters in this collection are all Tukey equivalent to $\mathcal{U}_1$.
\end{thm}

The next theorem follows from Theorem \ref{thm.TukeyU_1}.
This shows that  the structure of the Tukey types below $\mathcal{U}_1$ is analogous to the structure of the Rudin-Keisler types below $\mathcal{U}_1$ as proved by Laflamme  (see Theorem \ref{thm.Laflammethms} (3)).

\begin{thm}[Dobrinen/Todorcevic \cite{Dobrinen/Todorcevic11}]\label{cor.1Tpred}
If $\mathcal{V}$ is nonprincipal and $\mathcal{V}\le_T\mathcal{U}_1$,
then either 
$\mathcal{V}\equiv_T\mathcal{U}_1$,
or
$\mathcal{V}\equiv_T\mathcal{U}_0$.
\end{thm}

This paper builds on Theorem \ref{thm.PRR(1)}  and extends the aforementioned results of \cite{Dobrinen/Todorcevic11} for $\mathcal{R}_1$
to a new class of topological Ramsey spaces, denoted $\mathcal{R}_{\al}$, $2\le\al<\om_1$.
These spaces
are constructed in Section \ref{sec.tRsR_al}, based on infinitely wide, well-founded  trees $\bT_{\al}$.
The fact that $\al$ may now be infinite necessitates a new construction of the base trees $\bT_{\al}$ for the spaces $\mathcal{R}_{\al}$, using auxiliary structures $\bS_{\al}$ to preserve information about how the trees were built.
A new
analysis of the canonical equivalence relations is also necessary in this context.   
See Section \ref{sec.tRsR_al} for more discussion of these issues.
By an induction on $2\le\al<\om_1$ cycling through Sections \ref{sec.tRs} and \ref{sec.canonizationsRal},
each $\mathcal{R}_{\al}$ is proved to be a topological Ramsey space (Theorem \ref{thm.R_altRs}) and the main theorem of this paper, the Ramsey-classification Theorem \ref{thm.PRR_al} for $\mathcal{R}_{\al}$ is proved for each $2\le\al<\om_1$.

Associated to each of these spaces $\mathcal{R}_{\al}$ is a notion of an ultrafilter Ramsey for $\mathcal{R}_{\al}$, which we denote $\mathcal{U}_{\al}$.
As each space $\mathcal{R}_{\al}$ is forcing-equivalent to Laflamme's $\bP_{\al}$,
 the ultrafilters $\mathcal{U}_{\al}$  are the same as the ultrafilters  forced by Laflamme.
In Theorem \ref{thm.TukeyU_al}  in  Section \ref{sec.R_alTukey}, we extend Theorem \ref{thm.TukeyU_1}
to classify all the isomorphism classes of ultrafilters which are Tukey reducible to $\mathcal{U}_{\al}$, for all $2\le\al<\om_1$.
These turn out to be exactly the countable iterations of Fubini products of ultrafilters obtained   as projections of $\mathcal{U}_{\al}$ via canonical equivalence relations.
Finally, in Theorem \ref{cor.1Tpred}, we show that the Tukey types of all ultrafilters Tukey reducible to $\mathcal{U}_{\al}$
forms a strictly decreasing chain of order type $\al+1$. 
This shows that the structure of the Tukey types of ultrafilters Tukey-reducible to $\mathcal{U}_{\al}$ is analogous to the structure of the  isomorphism types of ultrafilters Rudin-Keisler reducible to $\mathcal{U}_{\al}$ found by Laflamme. 
For ease of reading, we include basic definitions and theorems for topological Ramsey spaces  in Section \ref{sec.stevobook}.


\section{Definitions and Theorems for Topological Ramsey Spaces}\label{sec.stevobook}

The  background in this section  can be found in detail in  Chapter 5, Section 1 of \cite{TodorcevicBK10}, which we include for 
 the convenience of the reader.
The  axioms \bf A.1 \rm - \bf A.4 \rm
are defined for triples
$(\mathcal{R},\le,r)$
of objects with the following properties.
$\mathcal{R}$ is a nonempty set,
$\le$ is a quasi-ordering on $\mathcal{R}$,
 and $r:\mathcal{R}\times\om\ra\mathcal{AR}$ is a mapping giving us the sequence $(r_n(\cdot)=r(\cdot,n))$ of approximation mappings, where
$\mathcal{AR}$ is  the collection of all finite approximations to members of $\mathcal{R}$.
For $a\in\mathcal{AR}$ and $A,B\in\mathcal{R}$,
\begin{equation}
[a,B]=\{A\in\mathcal{R}:A\le B\mathrm{\ and\ }(\exists n)\ r_n(A)=a\}.
\end{equation}

For $a\in\mathcal{AR}$, let $|a|$ denote the length of the sequence $a$.  Thus, $|a|$ equals the integer $k$ for which $a=r_k(a)$.
For $a,b\in\mathcal{AR}$, $a\sqsubseteq b$ if and only if $a=r_m(b)$ for some $m\le |b|$.
$a\sqsubset b$ if and only if $a=r_m(b)$ for some $m<|b|$.
For each $n<\om$, $\mathcal{AR}_n=\{r_n(A):A\in\mathcal{R}\}$.
If $n>|a|$, then  $r_n[a,A]$ is the collection of all $b\in\mathcal{AR}_n$ such that $a\sqsubset b$ and $b\le_{\mathrm{fin}} A$.
\vskip.1in

\begin{enumerate}
\item[\bf A.1]\rm
\begin{enumerate}
\item
$r_0(A)=\emptyset$ for all $A\in\mathcal{R}$.\vskip.05in
\item
$A\ne B$ implies $r_n(A)\ne r_n(B)$ for some $n$.\vskip.05in
\item
$r_n(A)=r_m(B)$ implies $n=m$ and $r_k(A)=r_k(B)$ for all $k<n$.\vskip.1in
\end{enumerate}
\item[\bf A.2]\rm
There is a quasi-ordering $\le_{\mathrm{fin}}$ on $\mathcal{AR}$ such that\vskip.05in
\begin{enumerate}
\item
$\{a\in\mathcal{AR}:a\le_{\mathrm{fin}} b\}$ is finite for all $b\in\mathcal{AR}$,\vskip.05in
\item
$A\le B$ iff $(\forall n)(\exists m)\ r_n(A)\le_{\mathrm{fin}} r_m(B)$,\vskip.05in
\item
$\forall a,b,c\in\mathcal{AR}[a\sqsubset b\wedge b\le_{\mathrm{fin}} c\ra\exists d\sqsubset c\ a\le_{\mathrm{fin}} d]$.\vskip.1in
\end{enumerate}
\end{enumerate}

$\depth_B(a)$ is the least $n$, if it exists, such that $a\le_{\mathrm{fin}}r_n(B)$.
If such an $n$ does not exist, then we write $\depth_B(a)=\infty$.
If $\depth_B(a)=n<\infty$, then $[\depth_B(a),B]$ denotes $[r_n(B),B]$.

\begin{enumerate}
\item[\bf A.3] \rm
\begin{enumerate}
\item
If $\depth_B(a)<\infty$ then $[a,A]\ne\emptyset$ for all $A\in[\depth_B(a),B]$.\vskip.05in
\item
$A\le B$ and $[a,A]\ne\emptyset$ imply that there is $A'\in[\depth_B(a),B]$ such that $\emptyset\ne[a,A']\sse[a,A]$.\vskip.1in
\end{enumerate}
\end{enumerate}

\begin{enumerate}
\item[\bf A.4]\rm
If $\depth_B(a)<\infty$ and if $\mathcal{O}\sse\mathcal{AR}_{|a|+1}$,
then there is $A\in[\depth_B(a),B]$ such that
$r_{|a|+1}[a,A]\sse\mathcal{O}$ or $r_{|a|+1}[a,A]\sse\mathcal{O}^c$.\vskip.1in
\end{enumerate}

The topology on $\mathcal{R}$ is given by the basic open sets
$[a,B]$.
This topology is called the {\em natural} or {\em Ellentuck} topology on $\mathcal{R}$;
it extends the usual metrizable topology on $\mathcal{R}$ when we consider $\mathcal{R}$ as a subspace of the Tychonoff cube $\mathcal{AR}^{\bN}$.
Given the Ellentuck topology on $\mathcal{R}$,
the notions of nowhere dense, and hence of meager are defined in the natural way.
Thus, we may say that a subset $\mathcal{X}$ of $\mathcal{R}$ has the {\em property of Baire} iff $\mathcal{X}=\mathcal{O}\cap\mathcal{M}$ for some Ellentuck open set $\mathcal{O}\sse\mathcal{R}$ and Ellentuck meager set $\mathcal{M}\sse\mathcal{R}$.

\begin{defn}[\cite{TodorcevicBK10}]\label{defn.5.2}
A subset $\mathcal{X}$ of $\mathcal{R}$ is {\em Ramsey} if for every $\emptyset\ne[a,A]$,
there is a $B\in[a,A]$ such that $[a,B]\sse\mathcal{X}$ or $[a,B]\cap\mathcal{X}=\emptyset$.
$\mathcal{X}\sse\mathcal{R}$ is {\em Ramsey null} if for every $\emptyset\ne [a,A]$, there is a $B\in[a,A]$ such that $[a,B]\cap\mathcal{X}=\emptyset$.

A triple $(\mathcal{R},\le,r)$ is a {\em topological Ramsey space} if every property of Baire subset of $\mathcal{R}$ is Ramsey and if every meager subset of $\mathcal{R}$ is Ramsey null.
\end{defn}

The following result is   Theorem
5.4 in \cite{TodorcevicBK10}.

\begin{thm}[Abstract Ellentuck Theorem]\label{thm.AET}\rm \it
If $(\mathcal{R},\le,r)$ is closed (as a subspace of $\mathcal{AR}^{\bN}$) and satisfies axioms {\bf A.1}, {\bf A.2}, {\bf A.3}, and {\bf A.4},
then every property of Baire subset of $\mathcal{R}$ is Ramsey,
and every meager subset is Ramsey null;
in other words,
the triple $(\mathcal{R},\le,r)$ forms a topological Ramsey space.
\end{thm}

\begin{defn}[\cite{TodorcevicBK10}]\label{defn.5.16}
A family $\mathcal{F}\sse\mathcal{AR}$ of finite approximations is
\begin{enumerate}
\item
{\em Nash-Williams} if $a\not\sqsubseteq b$ for all $a\ne b\in\mathcal{F}$;
\item
{\em Sperner} if $a\not\le_{\mathrm{fin}} b$ for all $a\ne b\in\mathcal{F}$;
\item
{\em Ramsey} if for every partition $\mathcal{F}=\mathcal{F}_0\cup\mathcal{F}_1$ and every $X\in\mathcal{R}$,
there are $Y\le X$ and $i\in\{0,1\}$ such that $\mathcal{F}_i|Y=\emptyset$.
\end{enumerate}
\end{defn}

The next theorem appears as  Theorem 5.17  in \cite{TodorcevicBK10}.

\begin{thm}[Abstract Nash-Williams Theorem]\label{thm.ANW}
Suppose $(\mathcal{R},\le,r)$ is a closed triple that satisfies {\bf A.1} - {\bf A.4}. Then
every Nash-Williams family of finite approximations is Ramsey.
\end{thm}

\begin{defn}\label{def.frontR1}
Suppose $(\mathcal{R},\le,r)$ is a closed triple that satisfies {\bf A.1} - {\bf A.4}.
Let $X\in\mathcal{R}$.
A family $\mathcal{F}\sse\mathcal{AR}$ is a {\em front} on $[0,X]$ if
\begin{enumerate}
\item
For each $Y\in[0,X]$, there is an $a\in \mathcal{F}$ such that $a\sqsubset Y$; and
\item
$\mathcal{F}$ is Nash-Williams.
\end{enumerate}
$\mathcal{F}$ is a {\em barrier}  if (1) and ($2'$) hold,
where
\begin{enumerate}
\item[(2$'$)]
$\mathcal{F}$ is Sperner.
\end{enumerate}
\end{defn}

\begin{rem}
Any front on a topological Ramsey space is Nash-Williams; hence is Ramsey, by Theorem \ref{thm.ANW}.
\end{rem}


\section{Construction of trees $\bT_{\al}$ and the spaces $(\mathcal{R}_{\al},\le_{\al},r^{\al})$, for $\al<\om_1$}\label{sec.tRsR_al}

By recursion on $\al<\om_1$, we  construct trees $\bT_{\al}$, related auxilliary structures $\bS_{\al}$, and maps $\tau_{\beta,\al}$, $\sigma_{\beta,\al}$, $\psi_{\al}$, for $\beta<\al$.
After completing this recursive definition, we then define the spaces $\mathcal{R}_{\al}$.
These spaces 
are modified versions of dense subsets of 
 the forcings $\bP_{\al}$ of Laflamme in \cite{Laflamme89}.
The main difference is that we pair down his forcings and 
use trees and related structures instead of finite sets 
in such a way as will produce  topological Ramsey spaces.
This allows us to apply   the theorems mentioned in Section \ref{sec.stevobook}.

The purpose of the $\bS_{\al}$ is several-fold. 
They aid in the precision of the definitions of members of $\mathcal{R}_{\al}$ while having the members of $\mathcal{R}_{\al}$   be well-founded trees (hence countable objects).
They also  provide a simple way of understanding the canonical equivalence relations in terms of downward closed subsets of the $\bS_{\al}(n)$.
This in turn makes clear the structures of the Rudin-Keisler types  and Tukey types  of all ultrafilters Rudin-Keisler or  Tukey reducible to $\mathcal{U}_{\al}$.
For $\al\ge \om$, $\bS_{\al}$ will not truly be a tree, but will have a tree-like structure under the ordering of $\subset$.
Downward closed subsets of $\bS_{\al}$ will be chains which are well-ordered by the reverse ordering $\supset$ on $\bS_{\al}$.
This may seem a bit strange at first, but the $\bS_{\al}$'s are in fact the correct structures, completely and precisely capturing the structure of the spaces $\mathcal{R}_{\al}$.

The maps $\psi_{\al}:\bS_{\al}\ra\bT_{\al}$ are to be thought of as projection maps, projecting the structure of $\bS_{\al}$ onto the tree $\bT_{\al}$.
For $\al<\om\cdot\om$,
$\tau_{\beta,\al}$ will be the projection map from  $\bT_{\al}$ to $\bT_{\beta}$ and $\sigma_{\beta,\al}$ will be the projection map from $\bS_{\al}$ to $\bS_{\beta}$.
For $\al\ge\om\cdot\om$, this will almost be the case: Properties $(\dagger)$ and $(\ddagger)$ below will be preserved.

Let $\mathcal{R}_0$ denote the Ellentuck space.
For the recursive construction of $\mathcal{R}_1$ from $\mathcal{R}_0$, 
it is useful to represent the Ellentuck space as a space of trees as follows.
Let  $\bT_0$ denote the tree ${}^{\le 1}\bN$ of height $1$ and infinite width.
The members of $\mathcal{R}_0$ are all infinite subtrees of $\bT_0$.
For $X,Y\in\mathcal{R}_0$, $Y\le_0 X$ iff $Y\sse X$.
Let $\bS_0=\bT_0$ and $\psi_0$ be the identity map from $\bS_0$ to $\bT_0$.

In order to accommodate the recursive definitions of the trees $\bT_{\al}$, $1\le\al<\om_1$,
we very slightly modify the definition of  $\bT_1$ from \cite{Dobrinen/Todorcevic11}
by changing $\bT_1(0)$ from $\{\lgl\rgl,\lgl 0\rgl,\lgl 0,0\rgl\}$ to $\{\lgl\rgl,\lgl 0\rgl\}$.
The structure $\bS_1$ here is exactly the structure $\bT_1$ from \cite{Dobrinen/Todorcevic11}.
The reader familiar with that paper will immediately see  that this re-definition does not  change any of the results in there.
In fact, we could use  the same definition of $\bT_1$ here as in  \cite{Dobrinen/Todorcevic11} and just define all trees $\bT_n$ below to be exactly $\bS_n$, for all $n<\om$.
The shortcoming of that approach is that it will not lead to a recursive definition for $\bT_{\al}$, $\om\le \al<\om_1$.

\begin{defn}[$\bT_1$,  $\tau_{0,1}$, $\bS_1$, $\sigma_{0,1}$, $\psi_1$]\label{defn.T_1}
Let $l^0_0=0$, $l^0_1=1$, $l^0_2=3$,
and generally,
$l^0_{n+1}=l^0_n+n+1$, for $n\ge 2$.
Define 
\begin{equation}
\bT_1(0)=\{\lgl\rgl,\lgl 0\rgl\}.
\end{equation}
For $0<n<\om$, let 
\begin{equation}
\bT_1(n)=\{\lgl\rgl,\lgl n\rgl,\lgl n,i\rgl:l^0_n\le i<l^0_{n+1}\}.
\end{equation}
Let
\begin{equation}
\bT_1=\bigcup_{n<\om}\bT_1(n).
\end{equation}
Note that $\bT_1$ is a tree, ordered by end-extension, which is a substructure of ${}^{\le 2}\bN$.


Define $\tau_{0,1}:\bT_1\ra\bT_0$, the {\em projection of $\bT_1$ to $\bT_0$},
 by 
\begin{align}
\tau_{0,1}(\lgl 0\rgl)
&=\lgl 0\rgl;\cr
\tau_{0,1}(t)
&=\lgl t(1)\rgl, \mathrm{\ if\ } |t|=2;\cr  
\tau_{0,1}(t)
&=\lgl\rgl, \mathrm{\  if \ }t=\lgl n\rgl\mathrm{\ and\ } n\ne 0\mathrm{\ or\ } t=\lgl\rgl.\cr
\end{align}

Define the auxiliary structure $\bS_1$
as follows.
For each $n<\om$,
let $\bS_1(n)$ be the collection of functions with domain $\{0,1\}$, $\{1\}$, or $\emptyset$  such that 
if $0\in\dom(f)$, then $l^0_n\le f(0)<l^0_{n+1}$, and if $1\in\dom(f)$, then $f(1)=n$.
Let $\bS_1=\bigcup_{n<\om}\bS_1(n)$.
Note that $\bS_1$  forms a tree structure under extension.
For example, $\{(1,1)\}\subset\{(0,1),(1,1)\}$ in the extension ordering on $\bS_1$.
Define $\sigma_{0,1}:\bS_1\ra\bS_0$, the {\em projection of $\bS_1$ to $\bS_0$} by
$\sigma_{0,1}(s)= s\re\{0\}$, for each $s\in\bS_1$.


For each $n<\om$, there is a natural projection map $\psi_1:\bS_1\ra\bT_1$ such that for each $n<\om$, $\psi_1''\bS_1(n)=\bT_1(n)$.
This map 
is defined by 
\begin{align}
\psi_1(\{(0,0),(1,0)\})
&=\psi_1(\{(1,0)\})=\lgl 0\rgl;\cr
\psi_1(\{(0,i),(1,n)\})
&=\lgl n,i\rgl, \mathrm{\ for\ }  n\ge 1;\cr
\psi_1(\{(1,n)\})
&=\lgl n\rgl, \mathrm{\ for \ }n\ge 1; \cr
\psi_1(\{\emptyset\})
&=\lgl\rgl.\cr
\end{align}
\end{defn}

\begin{rem}
$\bS_1$ has a tree structure under the ordering $\subset$, but with the domain of the sequences reversed in order.
This is done so that it will be clear exactly how $\bS_{\al}$ is built from  $\bS_{\beta}$, for $\beta<\al$, and also to aid in understanding the Rudin-Keisler ordering on the ultrafilters $\mathcal{U}_{\al}$ Ramsey for  the spaces $\mathcal{R}_{\al}$.
\end{rem}

In preparation for the recursive construction,
assume  that we have fixed, for each limit ordinal $\al<\om_1$, a strictly increasing cofinal function $c_{\al}:\om\ra\al$.
For $\al=\om\cdot (n+1)$ for $ n<\om$, we may take  $c_{\al}:\om\ra\al$ to be given by
$c_{\al}(m)=\om\cdot n + m$.
Though not necessary, this does make the spaces $\bT_{\al}$, $\al<\om\cdot\om$ very clear.

Given that $\bT_{\beta}$ and $\bS_{\beta}$ have been defined,
we define the maps $\sigma_{\gamma,\beta}$ and $\tau_{\gamma,\beta}$ for all $\gamma<\beta$ as follows.
Define $\sigma_{\gamma,\beta}$ on $\bS_{\beta}$ by 
$\sigma_{\gamma,\beta}(s)=s\re(\gamma+1)$, for each $s\in\bS_{\beta}$. 
Hence, if $\dom(s)=[\zeta,\beta]$ with $\zeta\le \gamma$,
then $\sigma_{\gamma,\beta}(s)=s\re[\zeta,\gamma]$;
and if $\gamma<\zeta\le\beta$, then $\sigma_{\gamma,\beta}(s)=\lgl\rgl$.
Note that 
for each $t\in\bT_{\beta}$,
$\psi_{\gamma}\circ\sigma_{\gamma,\beta}\circ\psi_{\beta}^{-1}(t)$ is  a singleton.
(The singleton can be the set containing the empty sequence.)
Define 
 $\tau_{\gamma,\beta}(t)$ to be {\em the} member of $\psi_{\gamma}\circ\sigma_{\gamma,\beta}\circ\psi_{\beta}^{-1}(t)$.

By our choices of the functions $c_{\al}$ for $\al<\om\cdot\om$,
it follows that for all $\gamma<\beta<\om\cdot\om$,
$\sigma_{\gamma,\beta}:\bS_{\beta}\ra\bS_{\gamma}$ and 
$\tau_{\gamma,\beta}:\bT_{\beta}\ra\bT_{\gamma}$.
For $\gamma\ge\om\cdot\om$,
this will not necessarily be the case.
However, the following properties $(\dagger)$ and $(\ddagger)$ hold for $\beta=1$, and we will preserve them for all $\beta<\om_1$.
For $m<\om$, we shall let $\bS_{\beta}([m,\om))$ denote $\bigcup\{\bS_{\beta}(n):m\le n<\om\}$.

\begin{enumerate}
\item[$(\dagger)$]
For each $\gamma\le \beta$, there is a $k<\om$ such that for each $l\ge k$, there is an $m<\om$ such that 
$\bS_{\gamma}(l)\sse\sigma_{\gamma,\beta}(\bS_{\beta}(m))$. 
\end{enumerate}
In particular, there are $k,m<\om$ such that $\sigma_{\gamma,\beta}(\bS_{\beta}( [k,\om)))=\bS_{\gamma}([m,\om))$.

\begin{enumerate}
\item[$(\ddagger)$]
For each $\gamma\le \beta$, there is a $k<\om$ such that for each $l\ge k$, there is an $m<\om$ such that 
$\bT_{\gamma}(l)\sse\tau_{\gamma,\beta}(\bT_{\beta}(m))$. 
\end{enumerate}

\noindent\underbar{Induction Assumption for $1<\al<\om_1$}.
Let  $1<\al<\om_1$ and suppose that 
 for all $\beta<\al$ we have 
defined $\bT_{\beta},\bS_{\beta},\psi_{\beta}$,
and for all $\gamma<\beta<\al$, we have defined $\sigma_{\gamma,\beta},\tau_{\gamma,\beta}$
so that  $(\dagger)$ and $(\ddagger)$ hold.
\vskip.1in

There are two cases for the induction step: either $\al$ is a successor ordinal or else $\al$ is a limit ordinal.

\begin{defn}[$\bT_{\al}$,  $\bS_{\al}$,  $\psi_{\al}$, $\al$ a successor ordinal]\label{defn.T_{al}succ}
Suppose that $\al=\delta+k+1$, where $\delta$ is either $0$ or a countable limit ordinal and $k<\om$.
For $n\le k+1$, define $l^{\delta+k}_n=n$, and for 
$n\ge k+1$, define
$l^k_{n+1}=l^k_n +(n+1) -k$.
For each $n\le k$, let
\begin{equation}
\bT_{\al}(n)= \bT_{\delta+k}(n).
\end{equation} 
For each $n> k$, let
\begin{equation}
\bT_{\al}(n)=\{\lgl\rgl\}\cup\{\lgl n\rgl^{\frown}t:t\in\bigcup\{\bT_{\delta+k}(i):l^{\delta+k}_n\le i<l^{\delta+k}_{n+1}\}\}.
\end{equation} 
Let 
\begin{equation}
\bT_{\al}=\bigcup\{\bT_{\al}(n):n<\om\}.
\end{equation}

For each $n<\om$, 
define $\bS_{\al}(n)$ to consist of the empty set along with  all functions $f\re [\beta,\al]$, 
$\beta\le \al$, 
where $f=g\cup\{(\al,n)\}$ for some 
$g\in \bigcup\{\bS_{\delta+k}(l):l^{\delta+k}_n\le l< l^{\delta+k}_{n+1}\}$ with $\dom(g)=[0,\delta+k]$.
Let $\bS_{\al}=\bigcup_{n<\om}\bS_{\al}(n)$.

There is a natural projection map $\psi_{\al}:\bS_{\al}\ra\bT_{\al}$ such that for each $n<\om$, $\psi_{\al}''\bS_{\al}(n)=\bT_{\al}(n)$,
defined as follows.
Let $s\in\bS_{\al}(n)$.
If $\dom(s)=\emptyset$, then let $\psi_{\al}(s)=\lgl\rgl$.
If $\dom(s)=[\al,\al]$, then let $\psi_{\al}(s)=\lgl n\rgl$.
Now suppose $\dom(s)=[\zeta,\al]$ where $\zeta<\al$.
If $n\le k$, then
let
$\psi_{\al}(s) =\psi_{\delta+k}(s\re [\zeta,\delta+k])$.
If  $n> k$, then let
$\psi_{\al}(s)=\lgl n\rgl^{\frown}\psi_{\delta+k}(s\re [\zeta,\delta+k])$.
\end{defn}

\begin{defn}[$\bT_{\al}$,  $\bS_{\al}$,  $\psi_{\al}$, $\al$ a limit ordinal]\label{defn.T_al_lim}
For $n=0$, letting $\gamma=c_{\al}(0)<\beta=c_{\al}(1)<\al$,
by $(\ddagger)$ there is a $k_0$ such that for each $k\ge k_0$, there is an $m$ such that $\bT_{c_{\al}(0)}(k)\sse \tau_{c_{\al}(0),c_{\al}(1)}(\bT_{c_{\al}(1)}(m))$.
Choose the least such $k_0$ and fix $m_0$ such that $\bT_{c_{\al}(0)}(k_0)\sse 
\tau_{c_{\al}(0),c_{\al}(1)}(\bT_{c_{\al}(1)}(m_0))$
and let $l_0$ be the largest integer such that $\bT_{c_{\al}(0)}(l_0)\sse\bT_{c_{\al}(1)}(m_0)$.
For each $i\le l_0$, let 
\begin{equation}
\bT_{\al}(i)=\bT_{c_{\al}(0)}(i)
\end{equation}
Define $p_{-1}=0$ and $p_0=l_0$.

Assume we have defined $\bT_{\al}(i)$ for all $i\le p_n$ such that
\begin{enumerate}
\item
For each $p_{n-1}<i\le p_n$, $\bT_{\al}(i)=\bT_{c_{\al}(n)}(m)$ for some $m$;
and for some $l_n,m_n$:
\item
$\bT_{\al}(p_n)=\bT_{c_{\al}(n)}(l_n)$;
\item
$\bT_{c_{\al}(n)}(l_n)\sse\bT_{c_{\al}(n+1)}(m_n)$, and $l_n$ is the largest such;
\item
For all $q\ge l_n$, there is an $m$ such that $\bT_{c_{\al}(n)}(q)\sse\bT_{c_{\al}(n+1)}(m)$.
\end{enumerate}

Use $(\ddagger)$ to find a $k_{n+1}$ such that for each $q\ge k_{n+1}$, there is an $m$ such that $\bT_{c_{\al}(n+1)}(q)\sse\bT_{c_{\al}(n+2)}(m)$.
Choose the least such $k_{n+1}\ge m_n$ and fix $m_{n+1}$ such that 
$\bT_{c_{\al}}(n+1)\sse\bT_{c_{\al}}(n+2)$ and let $l_{n+1}$ be the largest integer such that 
$\bT_{c_{\al}(n+1)}(l_{n+1})\sse\bT_{c_{\al}(n+2)}(m_{n+1})$.
Put 
\begin{equation}
\bT_{\al}(i)=\bT_{c_{\al}(n+1)}(m_n+i-p_n),
\end{equation}
for
$i=p_n+1,\dots, p_n+l_{n+1}-m_n:=p_{n+1}$.
Let 
\begin{equation}
\bT_{\al}=\bigcup\{\bT_{\al}(j):j<\om\}.
\end{equation}
Note that $(\ddagger)$ is preserved up to and including $\al$ by this construction.

Define $\bS_{\al}$ to be the collection functions with domain $\al+1$ (ordered downwards) as follows.
For each $n<\om$ and $p_{n-1}<i\le p_n$,
let $\bS_{\al}(i)$ consist of the emptyset along with the collection of all functions 
$f$, satisfying
\begin{enumerate}
\item
$\dom(f)=[\beta,\al]$ for some $\beta\le\al$;
\item
 $f\re [\beta,c_{\al}(n)]\in\bS_{c_{\al}(n)}(m)$,
 where $m$ is such that  
 $\bT_{\al}(i)=\bT_{c_{\al}(n)}(m)$; 
  and
\item
$f\re[c_{\al}(n) +1,\al]$ is the constant function with value $i$.
\end{enumerate}
Then we set
\begin{equation}
\bS_{\al}=\bigcup_{i<\om}\bS_{\al}(i).
\end{equation}

There is a natural projection map $\psi_{\al}:\bS_{\al}\ra\bT_{\al}$ such that for each $n<\om$, $\psi_{\al}''\bS_{\al}(n)=\bT_{\al}(n)$.
For $i<\om$, $s\in\bS_{\al}(i)$ and $n$ such that $p_{n-1}<i\le p_n$,
define
$\psi_{\al}(s) = \psi_{c_{\al}(n)}\circ \sigma_{c_{\al}(n),\al}(s)$.
\end{defn}

If $s,s'\in\bS_{\al}$, $\dom(s)=[\beta,\al]$, $\dom(s')=[\beta',\al]$,
we say that $s'$ is an {\em immediate successor} of $s$ iff $\beta=\beta' +1$ and $s'\supset s$; we also say that $s$ is the {\em immediate predecessor} of $s'$.
We shall say that $s$ is a {\em splitting node} iff $\beta$ is a successor ordinal, say $\beta=\gamma+1$, and  there are 
$s_0,s_1\in\bS_{\al}$ with $\dom(s_0)=\dom(s_1)=[\gamma,\al]$,
$s_0\re[\beta,\al]=s_1\re[\beta,\al]=s$,
and $s_0\ne s_1$ (that is, $s_0(\gamma)\ne s_1(\gamma)$).

Note that for each $t\in \bT_{\al}$, $\psi^{-1}_{\al}(t)$ is a closed  interval  of $\bS_{\al}(n)$ and the maximal node in $\psi^{-1}_{\al}(t)$ is either maximal in $\bS_{\al}$ or else a splitting node in $\bS_{\al}$. 
Whenever $s$ is  a splitting node in $\bS_{\al}$,
$\min(\dom(s))$ must be a successor ordinal.
This allows us to define 
the lexicographic ordering $<_{\mathrm{lex}}$ on $\bS_{\al}$.

\begin{defn}\label{def.lex}
For $s,s'\in\bS_{\al}$, define $s<_{\mathrm{lex}} s'$ iff
either $s\subsetneq s'$ (i.e.\ $s'$ properly extends $s$),
or else $s(\beta-1)<s'(\beta-1)$,
where $\beta\le\al$ is the maximal ordinal such that 
$s\re[\beta,\al]=s'\re[\beta,\al]$ and $s(\beta-1)\ne s'(\beta-1)$.
By {\em isomorphism} between substructures of $\bS_{\al}$, we mean a bijection which preserves  the lexicographical order.
\end{defn}

\begin{rem}
Each  $\bS_{\al}$ forms a tree-like structure.
For $n<\om$,  $\bS_n$ truly  is  a tree.
For each $s\in\bS_{\al}$, $\{s'\in\bS_{\al}:s'\subset s\}$
forms a linearly ordered set which is well-ordered by $\supset$.
Moreover, for each $n<\om$, there are only finitely many splitting nodes  in $\bS_{\al}(n)$.
The $\bS_{\al}$ may be viewed as  the true structures,  the trees $\bT_{\al}$ being obtained by the simple projection mappings $\psi_{\al}:\bS_{\al}\ra\bT_{\al}$.
The map $\psi_{\al}$ essentially glues   all non-splitting nodes  between two consecutive splitting nodes of $\bS_{\al}$ to the upper splitting node.
\end{rem}

We are now equipped to define $\mathcal{R}_{\al}$.

\begin{defn}[$(\mathcal{R}_{\al},\le_{\al},r^{\al})$, $1\le\al<\om_1$]\label{defn.R_al}
A subset $X\sse\bT_{\al}$ is a member of $\mathcal{R}_{\al}$ iff
$\psi^{-1}_{\al}(X)\cong \bS_{\al}$.
Equivalently, 
$X\in\mathcal{R}_{\al}$ iff
there is a strictly increasing sequence $(k_n)_{n<\om}$ such that 
\begin{enumerate}
\item
$X\cap\bT_{\al}(m)\ne\emptyset$ iff $m=k_n$ for some $n<\om$;
\item
For each $n<\om$,
$\psi_{\al}^{-1}(X\cap\bT_{\al}(k_n))\cong \bS_{\al}(n)$.
\end{enumerate}
For the sequence $(k_n)_{n<\om}$ above,  
we let $X(n)$ denote $X\cap \mathbb{T}_{\al}(k_n)$.
We shall call $X(n)$ the {\em $n$-th tree} of $X$.
For each $n<\om$, 
\begin{equation}
\mathcal{R}_{\al}(n)=\{X(n):X\in\mathcal{R}_{\al}\}.
\end{equation} 
For $n<\om$, $r^{\al}_n(X)$ denotes $\bigcup_{i<n}X(i)$.
The set of $n$-th approximations to members in $\mathcal{R}_{\al}$ is
\begin{equation}
\mathcal{AR}^{\al}_n=\{r^{\al}_n(X):X\in\mathcal{R}_{\al}\},
\end{equation}
 and the set of all finite approximations to members in $\mathcal{R}_{\al}$ is
\begin{equation}
\mathcal{AR}^{\al}=\bigcup_{n<\om}\mathcal{AR}^{\al}_n.
\end{equation}
For $X,Y\in\mathcal{R}_{\al}$, define $Y\le_{\al} X$ iff 
there is a strictly increasing sequence $(k_n)_{n<\om}$ such that for each $n<\om$, $Y(n)\sse X(k_n)$.

Let $a,b\in\mathcal{AR}^{\al}$ and $A,B\in\mathcal{R}_{\al}$.
The quasi-ordering $\le^{\al}_{\mathrm{fin}}$ on $\mathcal{AR}^{\al}$ is defined as follows:
$b\le^{\al}_{\mathrm{fin}} a$ if and only if there are $n\le m$ 
such that
$a\in\mathcal{AR}^{\al}_m$,
 $b\in\mathcal{AR}^{\al}_n$, and  
there is  a strictly increasing sequence $(k_i)_{i<n}$  with $k_{n-1}<m$ such that  
for each $i< n$,  $b(i)$ is a subtree of $a(k_i)$ (equivalently, $b(i)\sse a(k_i)$).
In fact, $\le^{\al}_{\mathrm{fin}}$ is a partial ordering.
We write $a\le^{\al}_{\mathrm{fin}} B$ if and only if there is an $A\in\mathcal{R}_{\al}$ such that $a\sqsubset A$ and  $A\le_{\al} B$.
$B/a$ is defined to be $\bigcup\{B(n):n\ge\depth_B(a)\}$.
The basic open sets are given by 
\begin{equation}
[a,B]=\{X\in\mathcal{R}_{\al}:a\sqsubseteq X\mathrm{\  and \ }X\le_{\al} B\}.
\end{equation}
\end{defn}

\begin{rem}\label{rem.barrier}
Since the quasi-ordering $\le_{\mathrm{fin}}^{\al}$ is actually a partial ordering, it follows from Corollary 5.19 in \cite{TodorcevicBK10} that for any front $\mathcal{F}$ on $[0,X]$,
$X\in\mathcal{R}_{\al}$, there is a $Y\le_{\al} X$ for which $\mathcal{F}|Y$ is a barrier.
\end{rem}

We point out the following trivial but useful facts.

\begin{fact}
\begin{enumerate}
\item
For  $u\sse\bT_{\al}$,
$u\in\mathcal{R}_{\al}(n)$ iff $\psi_{\al}^{-1}(u)\sse\bS_{\al}(m)$  for some  $m\ge n$ and $\psi_{\al}^{-1}(u)\cong \bS_{\al}(n)$.
\item
 $u\in\mathcal{R}_{\al}(n)$ iff the structure  obtained by identifying each node $t$  in $u$ which is both not  a leaf and not a  splitting node in $u$ with the minimal splitting node in $u$ above  $t$,
is isomorphic to $\bT_{\al}(n)$.
\item
Because of the structure inherent in being a member of $\mathcal{R}_{\al}$, the following are equivalent for all $X,Y\in\mathcal{R}_{\al}$:
\begin{enumerate}
\item[(a)]
$Y\le_{\al} X$. 
\item[(b)]
There is a strictly increasing sequence $(k_n)_{n<\om}$ such that for each $n<\om$, $Y(n)$ is a subtree of $X(k_n)$, 
$\psi_{\al}^{-1}(Y(n))$ is isomorphic to $\bS_{\al}(n)$,
and $\psi_{\al}^{-1}(Y(n))$ is a substructure of $\psi_{\al}^{-1}(X(k_n))$.
\item[(c)]
$Y\sse X$.
\end{enumerate}
\end{enumerate}
\end{fact}

Throughout this paper, we use the following fact without further mention.

\begin{fact}\label{fact.B(n)}
Suppose $1\le\al<\om_1$, $n<\om$, $a\in\mathcal{AR}^{\al}_n$,  $B\in\mathcal{R}_{\al}$, and there are $k<k'$ such that
$B(n)\sse\mathbb{T}_{\al}(k')$ and $a(n-1)\sse\mathbb{T}_{\al}(k)$.
Then $a\cup (B/r_n^{\al}(B))$ is a member of $\mathcal{R}_{\al}$.
\end{fact}





\section{$\mathcal{R}_{\al}$ is a topological Ramsey space, for each $\al<\om_1$}\label{sec.tRs}

In this section, we prove by induction  that each $\mathcal{R}_{\al}$, $2\le\al<\om_1$, is a topological Ramsey space.
In the process, we define the canonical equivalence relations on $\mathcal{R}_{\al}(n)$ and on $\mathcal{AR}^{\al}_n$.
Recall that $\mathcal{R}_0$ denotes the Ellentuck space, which is the fundamental example of a topological Ramsey space.
In Theorem 3.9 of \cite{Dobrinen/Todorcevic11}, $\mathcal{R}_1$ was shown to be a topological Ramsey space.
This forms the basis of the induction scheme which cycles through this and the next section.
We begin this section by setting the stage for the introduction of the canonical equivalence relations.

A subset $S\sse\bS_{\al}$ is called {\em downward closed} iff $\emptyset \in S$ and, for all  $s\in S$, if $\dom(s)=[\beta,\al]$,
then also $s\re [\gamma,\al]\in S$ for all $\gamma\in[\beta,\al]$.
Two downward closed sets $S,S'\sse\bS_{\al}$ are  {\em isomorphic} iff  there is a bijection between $S$ and $S'$ which preserves the lexicographic ordering.

\begin{defn}\label{defn.frakS(n,m)}
For each $n<\om$,
define $\mathfrak{S}_{\al}(n)$ to be the collection of all non-empty downward closed subsets of $\bS_{\al}(n)$.
For each $n\le m$, 
$\mathcal{R}_{\al}(n)|\bT_{\al}(m)$ denotes the collection of all $u\in\mathcal{R}_{\al}(n)$ such that $u\sse\bT_{\al}(m)$.
Define $S\in\mathfrak{S}_{\al}(n,m)$ iff $S\in\mathfrak{S}_{\al}(n)$ and 
there is a $u\in\mathcal{R}_{\al}(n)|\bT_{\al}(m)$ and a nonempty subtree $v\sse u$ such that $S\cong\psi^{-1}_{\al}(v)$.
\end{defn}

We point out the following.
The set $\{\emptyset\}$ is the $\sse$-minimal member of each $\mathfrak{S}_{\al}(n)$; 
$\{\lgl\rgl\}$ is the smallest nonempty subtree of any member in  $\mathcal{R}_{\al}(n)$.
 $\psi^{-1}_{\al}(v)=\{\emptyset\}$ iff  $v=\{\lgl\rgl\}$.
Note that  $\mathfrak{S}_{\al}(n,m)$ is finite, for all $n\le m$.  However, if $\al$ is infinite, then $\mathfrak{S}_{\al}(n)$ is countably infinite.

Given $\beta\le\al$, we shall let $S^{\al}_{\beta}$, or just $S_{\beta}$, denote the member of $\mathfrak{S}_{\al}(0)$ which is a downward closed chain of order type $[\beta,\al+1]$.
Thus,  $S_{\beta}=\bS_{\al}(0)\re[\beta,\al+1]$, which is the collection of all constantly zero functions on domains $[\gamma,\al]$, for $\beta\le\gamma\le\al$, along with the empty function.
The next fact follows immediately from Definition \ref{defn.frakS(n,m)}.

\begin{fact}\label{fact.frakSm<m'}
Let $n\le m<m'$. 
\begin{enumerate}
\item
$\mathfrak{S}_{\al}(n,m)\sse\mathfrak{S}_{\al}(n,m')\sse\mathfrak{S}_{\al}(n)$.
\item
$\mathfrak{S}_{\al}(n)=\bigcup\{\mathfrak{S}_{\al}(n,m):m\ge n\}\cup\{\{(\al,n),\emptyset\}\}$.
\end{enumerate}
\end{fact}

Next we define projection maps $\pi_S$.  The map $\pi_S$ takes a structure $u$ in its domain and projects it to the substructure of $u$ whose $\psi_{\al}$-preimage is isomorphic $S$.

\begin{defn}\label{defn.pi_S}
Let $1\le\al<\om_1$ and $m<\om$ be given.
Let $S\in\mathfrak{S}_{\al}(m)$.
Define $\pi_S$ on $\mathcal{R}_{\al}(m)$ as follows:
Given $u\in\mathcal{R}_{\al}(m)$,
let $\iota_u:\bS_{\al}(m)\ra\psi_{\al}^{-1}(u)$ be {\em the} isomorphism from   $\bS_{\al}(m)$ to $\psi^{-1}_{\al}(u)$.
Define
\begin{equation}
\pi_S(u)=
\psi_{\al}\circ\iota_u (S).
\end{equation}
Given $n<m$, letting $S$ be the subset of $\bS_{\beta}(m)$ which  consists of the lexicographically least (i.e.\ leftmost) members of $\bS_{\al}(m)$ which together comprise a set isomorphic to $\bS_{\al}(n)$,
let $\pi_{m,n}^{\al}$ denote $\pi_S$ for this particular $S$.
\end{defn}

Note that if $n<m$ and $S$ is any downward closed subset of $\bS_{\al}(m)$ such that $S$ is isomorphic to $\bS_{\al}(n)$,
then $\pi_S$ is in fact a map from $\mathcal{R}_{\al}(m)$ to $\mathcal{R}_{\al}(n)$.

We now introduce the various canonical equivalence relations.

\begin{defn}[Canonical Equivalence Relations on $\mathcal{R}_{\al}(n)$, for $\al<\om_1$]\label{defn.canonicaleqrels}
For each $n<\om$, 
each  $S\in \mathfrak{S}_{\al}(n)$ induces the equivalence relation $\E_S$ on $\mathcal{R}_{\al}(n)$ defined as follows:
For $u,v\in\mathcal{R}_{\al}(n)$,
\begin{equation}
u \E_S v \Leftrightarrow
\pi_S(u)=\pi_S(v).
\end{equation}
Let $\mathcal{E}_{\al}(n)$ denote the collection of all equivalence relations of the form $\E_S$, where $S\in\mathfrak{S}_{\al}(n)$.
$\mathcal{E}_{\al}(n)$ is the set of {\em canonical equivalence relations on $\mathcal{R}_{\al}(n)$}.
\end{defn}

\begin{defn}[Canonical Equivalence Relations on $\mathcal{R}_{\al}(n)|X(m)$, for $\al<\om_1$, $X\in\mathcal{R}_{\al}$, and $n\le m<\om$]\label{defn.canonicaleqrelsnm}
Given any $\al<\om_1$, $X\in\mathcal{R}_{\al}$, and $n\le m$, 
the
 {\em canonical equivalence relations on $\mathcal{R}_{\al}(n)|X(m)$} are given by $\E_S$, 
where $S\in\mathfrak{S}_{\al}(n,m)$.  
\end{defn}

\begin{rem}
For any $n\le m$ and any $S\in\mathfrak{S}_{\al}(n)$,
there is an $S'\in\mathfrak{S}_{\al}(n,m)$ such that $\E_S$ is the same as $\E_{S'}$ when restricted to 
 $\mathcal{R}_{\al}(n)|X(m)$.
Moreover, this $S'$ is unique, and  it must be the case that $S\sse S'$.
\end{rem}

\begin{defn}[Canonical Equivalence Relations on $\mathcal{AR}^{\al}_n$]\label{defn.canon-chunk} 
For any given $n_0<n_1<\om$ and $X\in\mathcal{R}_{\al}$,  let
$X[n_0,n_1)=\bigcup\{X(n):n_0\le n<n_1\}$.
Let
\begin{equation} 
\mathcal{R}_{\al}[n_0,n_1)
=\bigcup\{ X[n_0,n_1): X\in\mathcal{R}_{\al}\}.
\end{equation}
We shall say that an equivalence relation $\E$ on  $\mathcal{R}_{\al}([n_0,n_1))$  is {\em canonical} iff there   are    $S(i)\in\mathfrak{S}_{\al}(i)$, 
 $n_0\le i<n_1$, such that for all $x,y\in\mathcal{R}_{\al}([n_0,n_1))$,
\begin{equation} 
x\E y\Leftrightarrow \forall n_0\le i<n_1,\ x(i) \E_{S(i)} y(i).
\end{equation}

Taking $n_0=0$ and $n_1=n$, this
 defines the canonical equivalence relations on $\mathcal{AR}^{\al}_n$, for all $\al<\om_1$.
\end{defn}

\bf Numbers of Canonical Equivalence Relations. \rm
For each $k,n<\om$, the number of canonical equivalence relations on $\mathcal{R}_k(n)$  and $\mathcal{AR}^k_n$ are given by a recursive formula.
Let $N_k(n)$ denote the number of canonical equivalence relations on $\mathcal{R}_k(n)$.
Recall from \cite{Dobrinen/Todorcevic11} that for each $n$,
 $N_1(n)=2^{n+1}+1$, and for $n\ge 1$, there are $\Pi_{i<n}(2^{i+1}+1)$ canonical equivalence relations on $\mathcal{AR}^1_n$.
 It will be proved in Section \ref{sec.canonizationsRal}  that the canonical equivalence relations on $\mathcal{R}_{\al}(n)$ and $\mathcal{AR}^{\al}_n$ are precisely the ones defined above.
Hence, for $k\ge 1$, 
\[N_{k+1}(n)=\left\{
\begin{array}{l l}
N_k(n)+1 & \quad \text{if $n\le k$}\\
(\Pi_{l^k_n\le j<l^k_{n+1}}N_k(j))+1 &\quad\text{if $n>k$}\\
\end{array} \right.
\]
For $n\ge 1$, there are $\Pi_{i<n}N_k(i)$ many canonical equivalence relations on $\mathcal{AR}_{k}^n$.

Thus, for $k=2$,
there are $4$ canonical equivalence relations on $\mathcal{R}_2(0)$; $6$ canonical equivalence relations on $\mathcal{R}_2(1)$; $154$ canonical equivalence relations on $\mathcal{R}_2(2)$; etc.
There are $4$ canonical equivalence relations on $\mathcal{AR}^{\al}_1$; $24$ canonical equivalence relations on $\mathcal{AR}^{\al}_2$; $3696$ canonical equivalence relations on $\mathcal{AR}^2_3$; etc.

For $\om\le\al<\om_1$ and $n\le m$, $\mathfrak{S}_{\al}(n,m)$ is finite; however, $\mathfrak{S}_{\al}(n)$ is countably infinite.
\vskip.1in

The following theorem for  $\mathcal{R}_1$  was proved in \cite{Dobrinen/Todorcevic11}.
Recall that $\mathcal{AR}^1_n|D$ denotes the collection of all $a\in\mathcal{AR}^1_n$ such that $a\le^1_{\mathrm{fin}} D$.

\begin{thm}[Canonization Theorem for $\mathcal{AR}^1_n$ \cite{Dobrinen/Todorcevic11}]\label{thm.original}
Let  $1\le n<\om$.
Given any $A\in\mathcal{R}_1$ and any equivalence relation $\R$ on  $\mathcal{AR}^1_n|A$,
there is a $D\le_1 A$ such that $\R$ is canonical on $\mathcal{AR}^1_n|D$.
\end{thm}

Theorem \ref{thm.original} serves as the basis for the following {\bf Inductive Scheme:}
Given Theorem \ref{thm.original}, we prove Theorem \ref{thm.FCTR^beta_n}  and
Lemma \ref{lem.n_0n_1} for $\beta=1$.
These are then used to prove Theorems \ref{thm.FCT_R_al(n)}, \ref{thm.PigeonR_al(n)}, and \ref{thm.R_altRs} for
for $\al=2$.
Given these theorems, we then 
prove 
Theorems \ref{thm.PRR_al} and \ref{thm.FCTR^al_n} in Section \ref{sec.canonizationsRal} for $\al=2$.
The induction scheme continues for  $3\le\al<\om_1$ as follows.
Assume Theorems  \ref{thm.FCTR^al_n} and  \ref{thm.FCT_R_al(n)} hold for all $1\le \beta<\al$.
If $\al$ is a successor ordinal, say $\al=\beta+1$, then
we also assume Theorem \ref{thm.FCTR^beta_n} and
Lemma \ref{lem.n_0n_1} hold
for all $1\le\gamma<\beta$, and
we prove Theorem \ref{thm.FCTR^beta_n} and
Lemma \ref{lem.n_0n_1} hold
for  $\beta$.
If $\al$ is a limit ordinal, then 
by the time we have proved 
Theorems  \ref{thm.FCTR^al_n} and  \ref{thm.FCT_R_al(n)}
 for all $1\le\beta<\al$, 
we will also have proved 
Theorem \ref{thm.FCTR^beta_n} and
Lemma \ref{lem.n_0n_1}
for all $1\le\beta<\al$. 
These are then used to prove Theorems \ref{thm.FCT_R_al(n)}, \ref{thm.PigeonR_al(n)}, and \ref{thm.R_altRs} for
 $\al$, so that in particular, $\mathcal{R}_{\al}$ is a topological Ramsey space.
 Then 
 we prove 
  Theorems 
\ref{thm.PRR_al} and \ref{thm.FCTR^al_n} for $\al$ in Section \ref{sec.canonizationsRal}.

Thus, let $1<\al<\om_1$.
 In order to prove that $\mathcal{R}_{\al}$ is a topological Ramsey space, 
we will need to show that the Pigeonhole Principal \bf A.4 \rm holds for $\mathcal{R}_{\al}(n)$, for each $n<\om$. 
Toward this end, we first prove some finite canonization theorems.
The next theorem follows from Theorem \ref{thm.original} for $\beta=1$; for $\beta\ge 2$, it follows from Theorem \ref{thm.FCTR^al_n} for $\beta$.
We omit the  proof, as it is completely analogous to the standard proof of the Finite Ramsey Theorem from the Infinite Ramsey Theorem.

\begin{thm}[Finite Canonization Theorem for $\mathcal{AR}^{\beta}_n$]\label{thm.FCTR^beta_n}
For each $n\le k<\om$ and each $X\in\mathcal{R}_{\beta}$, there is an $m<\om$ such that 
for each equivalence relation $\E$ on $\mathcal{AR}^{\beta}_n|r^{\beta}_m(X)$,
there is an $a\in\mathcal{AR}^{\beta}_k|r^{\beta}_m(X)$ such that 
$\E$ is canonical on $\mathcal{AR}^{\beta}_n|a$.
\end{thm}

\begin{lem}\label{lem.n_0n_1}
Let $n_0< n_1$ and $k_0< k_1$ be such that $k_0\ge n_0$ and $k_1-k_0\ge n_1-n_0$,
 and let $X\in\mathcal{R}_{\beta}$.
There is an $m$ such that for each equivalence relation $\E$ on 
$\mathcal{R}_{\beta}[n_0,n_1)|r^{\beta}_m(X)$,
there is a $y\in\mathcal{R}^{\beta}[k_0,k_1)|r^{\beta}_m(X)$ such that $\E$ is canonical on 
$\mathcal{R}_{\beta}[n_0,n_1)|y$.
\end{lem}

\begin{proof}
Let $n_0,n_1,k_0,k_1$  be  as in the hypotheses.
Take $m$ from Theorem \ref{thm.FCTR^beta_n}
for $n_1$ and $k_1$.
Let $\E$ be an equivalence relation on 
$\mathcal{R}^{\beta}[n_0,n_1)|r^{\beta}_m(X)$.
Define an equivalence relation $\E'$ on $\mathcal{AR}^{\beta}_{n_1}|r^{\beta}_m(X)$ by defining $a \E' b $ if and only if 
$a[n_0,n_1) \E
 b[n_0,n_1)$, for  $a,b\in \mathcal{AR}^{\beta}_{n_1}|  r^{\beta}_m(X)$.
Then there is a $c\in \mathcal{AR}^{\beta}_{k_1}|r^{\beta}_m(X)$ such that $\E'$ is canonical on $\mathcal{AR}^{\beta}_{n_1}|c$.
Hence, $\E$ is canonical on 
$\mathcal{R}^{\beta}[n_0,n_1)| c[k_0,k_1)$.
\end{proof}

\begin{thm}[Finite Canonization Theorem for  $\mathcal{R}_{\al}(n)$]\label{thm.FCT_R_al(n)}
Let $n\le k<\om$ and $X\in\mathcal{R}_{\al}$ be given.
Then there is an $m$ such that for each equivalence relation $\E$ on  $\mathcal{R}_{\al}(n)|X(m)$,
there is a $y\in\mathcal{R}_{\al}(k)|X(m)$ such that $\E$ is canonical  on $\mathcal{R}_{\al}(n)|y$.
\end{thm}

\begin{proof}
Let $n,k,X$ be as in the hypotheses.
There are two cases.

\begin{case1} 
$\al$ is a successor ordinal.
\end{case1}

Let $\beta$ be such that $\al=\beta+1$.
Let $n_0=l^{\beta}_n$, $n_1=l^{\beta}_{n+1}$, $k_0=l^{\beta}_{k}$, and $k_1=l^{\beta}_{k+1}$.
Take $m_0$ from Lemma \ref{lem.n_0n_1}.
Let $m$ be large enough that $l^{\beta}_{m+1}-l^{\beta}_m>m_0$.
Let $\E$ be an equivalence relation on 
 $\mathcal{R}_{\al}(n)|X(m)$. 
Take $a\in\mathcal{AR}^{\beta}_{m_0}$ such that $a\sse\tau_{\beta,\al}''X(m)$. 
Let $\E'$ be the equivalence relation on 
 $\mathcal{R}_{\beta}[n_0,n_1)|a$ induced by $\E$ in the following manner: 
For all $u',v'\in  \mathcal{R}_{\beta}[n_0,n_1)|a$,
 $u' \E' v'$ iff
$u\E v$,
where $u=\{\lgl\rgl\}\cup\{\lgl m\rgl^{\frown}t:t\in u'\}$ and $v=\{\lgl\rgl\}\cup\{\lgl m\rgl^{\frown}t:t\in v'\}$.
By Lemma \ref{lem.n_0n_1}, there is a 
$y'\in\mathcal{AR}^{\beta}_{k_1}|a$ such that $\E'$ is canonical on 
$ \mathcal{R}_{\beta}[n_0,n_1)|y'[k_0,k_1)$,
given by some $S(i)\in\mathfrak{S}_{\beta}(i)$, $n_0\le i<n_1$.
Letting $y=\{\lgl\rgl\}\cup\{\lgl m\rgl^{\frown}t:t\in y'[k_0,k_1)\}$, we have that $y\in\mathcal{R}_{\al}(k)$.
Moreover, $\E$ is canonical on 
$\mathcal{R}_{\al}(n)|y$, given by $\E_S$,
where if at least one $S(i)\ne\{\emptyset\}$,
then we let $S=\{\emptyset\}\bigcup\{ s\cup\{(\al,n)\} :n_0\le i<n_1,\ s\in S(i)\}$,
and if all $S(i)=\{\emptyset\}$, then $S=\{\emptyset\}$.

\begin{case2}
$\al$ is a limit ordinal.
\end{case2}

Let $\gamma\le\delta$ and $n_{\gamma},k_{\delta}$ be the ordinals such that $\bT_{\al}(n)=\bT_{\gamma}(n_{\gamma})$ and $\bT_{\al}(k)=\bT_{\delta}(k_{\delta})$, 
by construction of $\bT_{\al}$.
$\bT_{\gamma}(n_{\gamma})$ is contained in $\tau_{\gamma,\delta}''\bT_{\delta}(n_{\delta})$, for some $n_{\delta}\le k_{\delta}$.
Note that necessarily $n_{\delta}\le k_{\delta}$.
Take $S\in\mathfrak{S}_{\delta}(n_{\delta})$ which satifies $\tau_{\gamma,\delta}\circ\pi_S(\bT_{\delta}(n_{\delta}))=\bT_{\gamma}(n_{\gamma})$.
Take $k'$ large enough that for any $w\in\mathcal{R}_{\delta}(k')$,
there is some member $v\in\mathcal{R}_{\delta}(k_{\delta})|w$ such that
as $u$ ranges over $\mathcal{R}_{\delta}(n_{\delta})|w$, their projections   $\tau_{\gamma,\delta}\circ\pi_S(u)$ range over (and possibly beyond) $\mathcal{R}_{\gamma}(n_{\gamma})|v$.
By Theorem \ref{thm.FCT_R_al(n)} for $\mathcal{R}_{\delta}$, 
there is an
$m$ such that for each 
$x\in\mathcal{R}_{\delta}(m)$ and
equivalence relation $\E'$ on  $\mathcal{R}_{\delta}(n_{\delta})|x$,
there is a $w\in\mathcal{R}_{\delta}(k')|x$ such that $\E'$ is canonical  on $\mathcal{R}_{\delta}(n_{\delta})|w$.

This $m$ works for $\mathcal{R}_{\al}$:
Let $X\in\mathcal{R}_{\al}$ and
 take any equivalence relation  $\E$ on 
$\mathcal{R}_{\al}(n)|X(m)$.
Take $x\sse \tau_{\delta,\al}(X(m))$ such that $x\in\mathcal{R}_{\delta}(m)$.
Let $\stem(x)$ denote the collection of all $t\in X(m)$ which are strictly below all nodes in $x$.
Note that $\stem(x)$ is a downward closed interval in $X(m)$.
Define $\E'$ to be the equivalence relation on $\mathcal{R}_{\delta}(n_{\delta})|x$ as follows.
For $y\in\mathcal{R}_{\delta}(n_{\delta})$,
 let $\bar{y}$ denote {\em the} member of $\mathcal{R}_{\al}(n)$ for which $\tau_{\delta,\al}(\bar{y})=y$.
For $y,z\in \mathcal{R}_{\delta}(n_{\delta})|x$, define $y\E' z$ iff $\bar{y}\E \bar{z}$.
By Theorem \ref{thm.FCT_R_al(n)},
there is a 
$w\in\mathcal{R}_{\delta}(k')|x$ such that $\E'$ is canonical  on $\mathcal{R}_{\delta}(n_{\delta})|w$.
By our choice of $k'$,
there is some member $v\in\mathcal{R}_{\delta}(k_{\delta})|w$ such that
as $u$ ranges over $\mathcal{R}_{\delta}(n_{\delta})|w$, their projections   $\tau_{\gamma,\delta}\circ\pi_S(u)$ range over (and possibly beyond) $\mathcal{R}_{\gamma}(n_{\gamma})|v$.
Let $\bar{v}=v\cup\stem(x)$.
Then $\bar{v}\in\mathcal{R}_{\al}(k)$, and $\E$ is canonical on $\mathcal{R}_{\al}(n)|\bar{v}$.
\end{proof}

\begin{thm}[Finite Version of the Pigeonhole Principal for $\mathcal{R}_{\al}(n)$]\label{thm.PigeonR_al(n)}
Let $n\le k<\om$ and $X\in\mathcal{R}_{\al}$ be given.
Then there is an $m$ such that for each $2$-coloring $f:\mathcal{R}_{\al}(n)|X(m)\ra 2$,
there is a $y\in\mathcal{R}_{\al}(k)|X(m)$ such that $f$ is monochromatic on $\mathcal{R}_{\al}(n)|y$.
\end{thm}

\begin{proof}
Let $n,k,X$ be as in the hypotheses.
Take $m$ from 
Theorem \ref{thm.FCT_R_al(n)}.
Then there is a 
$y\in\mathcal{R}_{\al}(k)|X(m)$ such that the equivalence relation induced by $f$ is canonical on $\mathcal{R}_{\al}(n)|y$.
But the only canonical equivalence relation induced by a $2$-coloring is the trivial one.
Thus, $f$ is monochromatic on $\mathcal{R}_{\al}(n)|y$.
\end{proof}

\begin{thm}\label{thm.R_altRs}
$(\mathcal{R}_{\al},\le_{\al},r^{\al})$ is a topological Ramsey space.
\end{thm}

\begin{proof}
By the Abstract Ellentuck Theorem,
it suffices to show that $(\mathcal{R}_{\al},\le_{\al},r^{\al})$  is a closed subspace of the Tychonov power $(\mathcal{AR}^{\al})^{\bN}$ of $\mathcal{AR}^{\al}$ with its discrete topology, and that $(\mathcal{R}_{\al},\le_{\al},r^{\al})$ satisfies axioms {\bf A.1} - {\bf A.4}.

$\mathcal{R}_{\al}$ is identified with the subspace of $(\mathcal{AR}^{\al})^{\bN}$  consisting of all sequences $\lgl a_n:n<\om\rgl$ such that there is an $A\in\mathcal{R}_{\al}$ such that for each $n<\om$,
$a_n=r^{\al}_n(A)$.
That $\mathcal{R}_{\al}$ is a closed subspace of 
$(\mathcal{AR}^{\al})^{\bN}$ follows from the fact that 
given any sequence $\lgl a_n:n<\om\rgl$ such that each $a_n\in\mathcal{AR}^{\al}_n$ and $r^{\al}_n(a_k)=a_n$ for each $k\ge n$, the union  $A=\bigcup_{n<\om}a_n$ is a member of $\mathcal{R}_{\al}$.
\bf A.1.  \rm and \bf A.2 \rm are trivially satisfied, by the definition of $\mathcal{R}_{\al}$.

\bf A.3. \rm
(1)  
If $\depth_B(a)=n<\infty$,
then  $a\le_{\mathrm{fin}}^{\al} r^{\al}_n(B)$.
If  $A\in[\depth_B(a),B]$,
then $r^{\al}_n(A)=r^{\al}_n(B)$ and for each $k\ge n$,
there is an $m_k$ such that $A(k)\sse B(m_k)$.
Letting $l$ be such that $a\in\mathcal{AR}^{\al}_l$,
for each $i\ge 1$,
let $w(l+i)$ be any subtree of $A(n+i)$
such that $\psi_{\al}^{-1}(w(l+i))\cong
\mathbb{S}_{\al}(l+i)$.
Let $A'=a\cup\bigcup\{w(l+i):i\ge 1\}$.
Then $A'\in[a,A]$, so $[a,A]\ne\emptyset$.

(2) Suppose $A\le_{\al} B$ and $[a,A]\ne\emptyset$.
Then $\depth_B(a)<\infty$ since $A\le_{\al} B$.
Let $n=\depth_B(a)$ and $k=\depth_A(a)$.
Note that $k\le n$ and for each $j\ge k$, $A(j)\sse B(l)$ for some $l\ge n$.
Let $A'=r^{\al}_n(B)\cup\bigcup\{A(n+i):i<\om\}$.
Then $A'\in[\depth_B(a),B]$ and 
$\emptyset\ne[a,A']\sse[a,A]$.

\bf A.4. \rm
Suppose that $\depth_B(a)=n<\infty$ and $\mathcal{O}\sse\mathcal{AR}^{\al}_{|a|+1}$.
Let $k=|a|$. 
(Recall that 
 $r^{\al}_{k+1}[a,B]$ is defined to be the collection of $c\in\mathcal{AR}^{\al}_{k+1}$ such that $r^{\al}_k(c)=r^{\al}_k(a)$ and $c(k)$ is a subtree of $B(m)$ for some $m\ge n$.)
So we may think of $\mathcal{O}$ as a 2-coloring on the collection of  subtrees $u\sse B(m)$, $m\ge n$,  such that $\psi_{\al}^{-1}(u)\cong\mathbb{S}_{\al}(k)$.
Say a set $u\in\mathcal{R}_{\al}(k)|B/r^{\al}_n(B)$ 
has color $0$ if $a\cup u$ is in $\mathcal{O}$ and has color $1$ if $a\cup u$ is in $\mathcal{O}^c$.
By repeated applications of Theorem  \ref{thm.PigeonR_al(n)},
we can construct an 
$A\in[\depth_B(a),B]$ such that 
either $r_{k+1}[a,A]\sse\mathcal{O}$,
 or else 
$r_{k+1}[a,A]\sse\mathcal{O}^c$.
\end{proof}


\section{Ramsey-classification theorems for $\mathcal{R}_{\al}$, $2\le\al<\om_1$}\label{sec.canonizationsRal}

This section contains the Ramsey-classification theorems  for equivalence relations on fronts on the  spaces $\mathcal{R}_{\al}$, $2\le \al<\om_1$.
Recall the Definitions \ref{defn.canonicaleqrels},
\ref{defn.canonicaleqrelsnm}, and
\ref{defn.canon-chunk} of the various canonical equivalence relations. 
We provide new facts here, not in \cite{Dobrinen/Todorcevic11}, necessitated by the fact that  $\al$ may be an infinite, countable ordinal.

\begin{fact}\label{fact.CanonThmEqReRestricted}
Let $n\le m<\om$ and $X\in\mathcal{R}_{\al}$, and suppose $\R$ is an equivalence relation on $\mathcal{R}_{\al}(n)$.
Then there is an $S\in\mathfrak{S}(n,m)$ and a $Y\le_{\al} X$ 
such that for each $y\in\mathcal{R}_{\al}(m)|Y$,
$\R\re(\mathcal{R}_{\al}(n)|y)$ is given by $\E_S$.
\end{fact}

\begin{proof}
For each $S\in\mathfrak{S}(n,m)$,
let 
\begin{equation}
\mathcal{X}_S=\{Y\le_{\al} X: \R\re(\mathcal{R}_{\al}(n)|Y(m))=\E_S\}.
\end{equation}
Since $\mathfrak{S}(n,m)$ is finite,
applying the
 Abstract Ellentuck Theorem and 
Theorem \ref{thm.FCT_R_al(n)} 
 for  $\mathcal{R}_{\al}(n)$,
we obtain an $S\in\mathfrak{S}(n,m)$ and 
a $Y\le_{\al} X$ such that $[\emptyset, Y]\sse\mathcal{X}_S$.
\end{proof}

It is useful to point out the following statement, which  can be  proved using $(\dagger)$ by a simple induction on $\al<\om_1$:
For every  $s\in\bS_{\al}(n)$ which has domain $[\beta,\al]$ for some $\beta<\al$,
there is an $n'>n$ such that any embedding of $\bS_{\al}(n)$ into $\bS_{\al}(n')$ sends $s$ to the immediate successor of a splitting node in $\bS_{\al}(n')$.
By an embedding, we mean an injective, lexicographic order-preserving map $\iota$, 
such that if $s'$ is an immediate predecessor of $s$, then $\iota(s')$ is an immediate predecessor of $\iota(s)$.

\begin{fact}\label{fact.frakS_mnseq}
Let $n\le m<m'$ and
$\R$ be an equivalence relation on $\mathcal{R}_{\al}(n)$.
Suppose that $S\in\mathfrak{S}_{\al}(n,m)$,
$S'\in\mathfrak{S}_{\al}(n,m')$, and
$X\in\mathcal{R}_{\al}$ satisfies 
$\R\re(\mathcal{R}_{\al}(n)|x)=\E_S$ for all $x\in\mathcal{R}_{\al}(m)|X$, 
and 
$\R\re(\mathcal{R}_{\al}(n)|x)=\E_{S'}$ for all $x\in\mathcal{R}_{\al}(m')|X$.
Then $S'\sse S$.
Moreover, given any embedding $\iota:\bS_{\al}(n)\ra\bS_{\al}(m)$, for
every $s\in S$ such that $\iota(s)$ is an immediate successor  of a splitting node in $\bS_{\al}(m)$, $s$ is also in $S'$.
\end{fact}

\begin{proof}
Assuming the hypotheses,
let $x\in\mathcal{R}_{\al}(m')|X$ and $z\in\mathcal{R}_{\al}(m)|x$.
Then for all $y,y'\in\mathcal{R}_{\al}(n)|z$,
we have that also $y,y'\in\mathcal{R}_{\al}(n)|x$.
Thus,
$y \E_{S} y'$ implies $y\R y'$, which in turn implies $y \E_{S'} y'$.
Hence, $S'\sse S$.

Suppose that there is an embedding  
 $\iota:\bS_{\al}(n)\ra\bS_{\al}(m)$
and an $s\in S\setminus S'$   such that
$\iota(s)$ is an immediate successor of some splitting node in $\bS_{\al}(m)$.
Then there are $x\in\mathcal{R}_{\al}(m')|X$
and  $y,y'\in\mathcal{R}_{\al}(n)|x$ such that $y\not{\E_S}\ y'$ but $y\E_{S'} y'$, contradiction.
\end{proof}

The next theorem will be essential in the proof of the main theorem, Theorem \ref{thm.PRR_al}.
Lemma \ref{lem.cohere} is included, as the argument there will be useful elsewhere.

\begin{thm}[Canonization Theorem for Equivalence Relations on $\mathcal{R}_{\al}(n)$]\label{thm.canon.eq.rel.R(n)}
Let $\R$ be an equivalence relation on $\mathcal{R}_{\al}(n)$ and let $X\in\mathcal{R}_{\al}(n)$.
Then there is an $S\in\mathfrak{S}_{\al}(n)$ and a  $Y\le_{\al}X$ such that $\R\re(\mathcal{R}_{\al}(n)|Y)$ is given by $\E_S$.
\end{thm}

\begin{proof}
Assume the hypotheses.
Recall the map
$\pi_{n+1,n}^{\al}:\mathcal{R}_{\al}(n+1)\ra\mathcal{R}_{\al}(n)$ from Definition \ref{defn.pi_S}. 
Let 
\begin{equation}
\mathcal{X}=\{X'\le_{\al} X:  X'(n) \R\  \pi_{n+1,n}^{\al}(X'(n+1)) \}.
\end{equation}
This set $\mathcal{X}$ will tell us whether or not the lexicographically least node in $X'(n)$ matters to the equivalence relation $\R$.
By the Abstract Ellentuck Theorem,
there is an $X'\le_{\al} X$ such that either $[\emptyset,X']\sse\mathcal{X}$, or $[\emptyset,X']\cap\mathcal{X}=\emptyset$.
Possibly thinning again, letting $S_{\al}$ denote $\{(\al,n),\emptyset\}\in\mathfrak{S}_{\al}(n)$,
we obtain a $Y\le_{\al} X'$ such that 
 either
 \begin{enumerate}
\item[(i)] 
for all $u,v\in \mathcal{R}_{\al}(n)|Y$,
$u\R v$; or
\item[(ii)]
 for all $u,v\in \mathcal{R}_{\al}(n)|Y$,
 if $u\R v$  then $\pi_{S_{\al}}(u)=\pi_{S_{\al}}(v)$.
\end{enumerate}
If case (i) holds, let $Z=Y$ and $S=\{\emptyset\}$.
In this case,  $u\R v$ for all $u,v\in\mathcal{R}_{\al}(n)|Z$.
Otherwise, case (ii) holds.

Suppose $\al<\om$.
Then $\bS_{\al}(n)$ is a finite tree, so $\mathfrak{S}_{\al}(n)$ is finite and consists of all finite subtrees of $\bS_{\al}(n)$.
Take $k>n$ large enough that $\mathfrak{S}_{\al}(n,k')=\mathfrak{S}_{\al}(n,k)$ for all $k'\ge k$.
For each  $S\in\mathfrak{S}_{\al}(n)\setminus\{\emptyset\}$,
define 
\begin{equation}
\mathcal{Y}_S=\{Y'\le_{\al} Y :
\forall u,v\in\mathcal{R}_{\al}(n)|Y'(2k)
 ( u\R v\mathrm{\ iff\ }
u \E_S v)\}.
\end{equation}
Let $\mathcal{Y}'=[\emptyset,Y]\setminus\bigcup_{S\in \mathfrak{S}_{\al}(n,k)}\mathcal{Y}_S$.
Then the $\mathcal{Y}_S$, $S\in \mathfrak{S}_{\al}(n)\setminus\{\emptyset\}$ along with $\mathcal{Y}'$ form an open  cover of
$[\emptyset,Y]$.
Since $\mathfrak{S}_{\al}(n)$ is finite, by the Abstract Ellentuck Theorem,  there is a $Z\le_{\al} Y$ such that either $[\emptyset,Z]\sse\mathcal{Y}_S$ for some  $S\in\mathfrak{S}_{\al}(n)\setminus\{\emptyset\}$, or else $[\emptyset,Z]\sse\mathcal{Y}'$.
By Theorem
\ref{thm.FCT_R_al(n)} for $\mathcal{R}_{\al}$,
it cannot be the case that $[\emptyset, Z]\sse\mathcal{Y}'$.
Since we are under the assumption that (ii) holds, there is some $S\in\mathfrak{S}_{\al}(n)\setminus\{\emptyset\}$ such that 
$\R\re(\mathcal{R}_{\al}(n)|Z)$ is given by $\E_S$.

Now suppose that $\om\le\al<\om_1$.
Then $\bS_{\al}(n)$ is not a tree, and $\mathfrak{S}_{\al}(n)$ is countably infinite.

\begin{lem}\label{lem.cohere}
Suppose $\om\le\al<\om_1$.
Let $n<\om$, $\R$ be an equivalence relation on $\mathcal{R}_{\al}(n)$, and $Y\in\mathcal{R}_{\al}$
such that (ii) holds; that is,
 for all $u,v\in \mathcal{R}_{\al}(n)|Y$,
 if $u\R v$  then $\pi_{S_{\al}}(u)=\pi_{S_{\al}}(v)$.
Then 
 there is a decreasing sequence, 
$Y=Y_{n}\ge_{\al} Y_{n+1}\ge_{\al}\dots$, 
and $S_m\in\mathfrak{S}_{\al}(n,m)$ for $m\ge n$ such that 
$S_n\contains S_{n+1}\contains\dots$ and 
for each $m\ge n$,
$\R\re(\mathcal{R}_{\al}(n)|z)$ is given by $\E_{S_m}$ for each $z\in \mathcal{R}_{\al}(m)| Y_m$.
Moreover, letting $Z=r^{\al}_n(Y)\cup\bigcup\{Y_n(n):n\ge m\}$ and $S=\bigcap\{S_m:m\ge n\}$,
we have that $\R\re(\mathcal{R}_{\al}(n)|Z)$ is given by $\E_S$.
\end{lem}

\begin{proof}
Assume the hypotheses.
Let $S_n=\bS_{\al}(n)$; this is the {\em only} member of $\mathfrak{S}_{\al}(n,n)$.
By Fact \ref{fact.CanonThmEqReRestricted},
there is a $Y_{n+1}\le_{\al} Y$ and an $S_{n+1}\in\mathfrak{S}(n,n+1)$ such that 
$\R\re\mathcal{R}_{\al}(n)|y$ is given by $\E_{S_{n+1}}$, for each $y\in\mathcal{R}_{\al}(n)|Y_{n+1}$.
Given $Y_m$, $m>n$,
by 
Fact \ref{fact.CanonThmEqReRestricted},
there is a $Y_{m+1}\le_{\al} Y_m$ and an $S_{m+1}\in\mathfrak{S}(n,m+1)$ such that 
$\R\re\mathcal{R}_{\al}(n)|y$ is given by $\E_{S_{m+1}}$, for each $y\in\mathcal{R}_{\al}(n)|Y_{m+1}$.
By Fact \ref{fact.frakS_mnseq}, $S_{m+1}\sse S_m$.

Let $Z=r^{\al}_n(Y)\cup\bigcup\{Y_m(m):m\ge n\}$, and 
let $S=\bigcap\{S_m:m\ge n\}$.
Then $S$ is downward closed, so $S\in\mathfrak{S}_{\al}(n)$.
Moreover, $S$ is nonempty, since the node $\{(\al,n)\}\in S_m$ for every $m\ge n$.
We claim that $\R\re\mathcal{R}_{\al}(n)|Z$ is given by $\E_S$.
Let $u,v\in\mathcal{R}_{\al}(n)|Z$, and let $m,m'$ be the integers such that $u\sse Z(m)$ and $v\sse Z(m')$.
If $m\ne m'$, then (ii) implies that $u\not\R v$.
Since $S$ is nonempty, also $u\not{\E}_S\ v$.
Now suppose $m=m'$.
If $u\R v$ then $u\E_{S_m} v$, which implies $u\E_S v$, since $S\sse S_m$.
If $u\not\R v$ then $u\not{\E}_{S_m} v$.
Let $s\in S_m$ be minimal in $S_m$ such that the copies of $s$ in $u$ and $v$ are different, under the isomorphisms of $\bS_{\al}(n)$ into $u$ and $v$.
Note that  $s$ must be the immediate successor of some splitting node in $\bS_{\al}(m)$.
But then 
 Fact \ref{fact.frakS_mnseq} implies
$s$ must be in $S_k$ for all $k\ge m$, which implies  $s\in S$, contradiction.
Hence, also $u\not{\E}_S\ v$.
Therefore, $u \R v$ iff $u\E_S v$.
\end{proof}

By the Lemma, the proof is complete.
\end{proof}

The following Lemmas \ref{lem.1}, \ref{lem.mixing} and 
\ref {claim.A}
were proved as 
Lemmas 4.6, 4.9, and 4.10, respectively,
in \cite{Dobrinen/Todorcevic11} for $\mathcal{R}_1$.
As the proofs are identical for all the spaces $\mathcal{R}_{\al}$, $1\le \al<\om_1$, we  restate these lemmas without proof.
In the following,  $X/(a,b)$ denotes $X/a\cap X/b$.

\begin{lem}\label{lem.1}
Suppose $1\le\al<\om_1$.

(1)
Suppose $P(\cdot,\cdot)$ is a property  such that for each $a\in\mathcal{AR}^{\al}$ and each $X\in\mathcal{R}_{\al}$, there is a $Z\le_{\al} X$ such that $P(a,Z)$ holds.
Then for each $X\in\mathcal{R}_{\al}$, 
there is a $Y\le_{\al} X$ such that for each $a\in\mathcal{AR}^{\al}|Y$ and each $Z\le_{\al} Y$, $P(a,Z/a)$ holds.

(2)
Suppose $P(\cdot,\cdot,\cdot)$ is a property such that for all $a,b\in\mathcal{AR}^{\al}$ and each $X\in\mathcal{R}_{\al}$, there is a $Z\le_{\al} X$ such that $P(a,b,Z)$ holds.
Then for each $X\in\mathcal{R}_{\al}$, there is a $Y\le_{\al} X$ such that for all $a,b\in\mathcal{AR}^{\al}|Y$ and all $Z\le_{\al} Y$, $P(a,b,Z/(a,b))$ holds.
\end{lem}

Given a front $\mathcal{F}$ on $[\emptyset,A]$ for some $A\in\mathcal{R}_{\al}$  and $f:\mathcal{F}\ra\bN$, we adhere to the following convention: 
If we write $f(a)$ or $f(a\cup u)$, 
it is assumed that $a,a\cup u$ are in $\mathcal{F}$.
Define
\begin{equation}
\hat{\mathcal{F}}=\{r_m^{\al}(a): a\in\mathcal{F},\ m\le n<\om, \mathrm{\ where\ } a\in\mathcal{AR}^{\al}_n\}.
\end{equation}
Note that $\emptyset\in\hat{\mathcal{F}}$, since $\emptyset=r_0^{\al}(a)$ for any $a\in\mathcal{F}$.
Recall that for $a\in\mathcal{AR}^{\al}_k$ and $m<n\le k$, $a[m,n)$ denotes $\bigcup\{a(i):i\in[m,n)\}$.
For any $X\le_{\al} A$,
define
\begin{equation}
\Ext(X)=\{a[m,n):
\exists m\le n\, (a\in \mathcal{AR}^{\al}_n,\mathrm{\ and\ } a[m,n)\sse X)\}.
\end{equation}
$\Ext(X)$ is the collection of all possible legal extensions into $X$.
Note that
$a[m,n)\sse X$
iff $a[m,n)\le^{\al}_{\mathrm{fin}} X$.
For any $a\in\mathcal{AR}^{\al}$,  let $\Ext(X/a)$ denote the collection of those $y\in \Ext(X)$ such that 
$y\sse X/a$.
Let $\Ext(X/(a,b))$  denote $\Ext(X/a)\cap\Ext(X/b)$.
For $u\in\Ext(X)$,
we write $v\in\Ext(u)$ 
to mean that 
$v\in\Ext(X)$ and $v\sse u$.

\begin{defn}\label{def.sepmix}
Fix $a,b\in\hat{\mathcal{F}}$ and $X\in\mathcal{R}_{\al}$.
We say that $X$ {\em separates $a$ and $b$} iff
for all  $x\in\Ext(X/a)$ and $y\in\Ext(X/b)$ such that $a\cup x$ and $b\cup y$ are in $\mathcal{F}$,
$f(a\cup x)\ne f(b\cup y)$.
We say that $X$ {\em mixes $a$ and $b$} iff there is no $Y\le_{\al} X$ which separates $a$ and $b$.
$X$ {\em decides for $a$ and $b$} iff either $X$ separates $a$ and $b$ or else $X$ mixes $a$ and $b$.

We say that $X/(a,b)$ {\em separates $a$ and $b$} iff
for all $x,y\in\Ext(X/(a,b))$ such that  $a\cup x$ and $b\cup y$ are in $\mathcal{F}$,
$f(a\cup x)\ne f(b\cup y)$.
 $X/(a,b)$ {\em mixes $a$ and $b$} iff there is no $Y\le_{\al} X/(a,b)$ which separates $a$ and $b$.
 $X/(a,b)$ {\em decides  for $a$ and $b$} iff either $X/(a,b)$ separates $a$ and $b$;
or else $X/(a,b)$ mixes $a$ and $b$.
\end{defn}

The following facts are useful to note.
 $X$  mixes $a$ and $b$  iff
 $X/(a,b)$  mixes $a$ and $b$
 iff for each $Y\le_{\al} X$, there are $x,y\in\Ext(Y)$ such that $f(a\cup x)=f(b\cup y)$
iff
 for all $Y\le_{\al} X$,
$Y$ mixes $a$ and $b$.
If $X$ separates $a$ and $b$ ($X/(a,b)$ separates $a$ and $b$),
then  for all $Y\le_{\al} X$ (for all $Y\le_{\al} X/(a,b)$),
$Y$ separates $a$ and $b$.
 $X/(a,b)$  decides for $a$ and $b$ iff either for all  $x,y\in\Ext(X/(a,b))$,
$f(a\cup x)\ne f(b\cup y)$,
or else there is no $Y\le_{\al} X/(a,b)$ which has this property.
Thus, $X/(a,b)$ mixes $a$ and $b$ 
iff $X$ mixes $a$ and $b$.
However, 
if  $X/(a,b)$ separates $a$ and $b$ it does not necessarily follow that $X$ separates $a$ and $b$.

\begin{lem}[Transitivity of Mixing]\label{lem.mixing}
For any $X\in\mathcal{R}_{\al}$ and any $a,b,c\in\hat{\mathcal{F}}$,
if $X$ mixes $a$ and $b$ and $X$ mixes $b$ and $c$,
then $X$ mixes $a$ and $c$.
\end{lem}

\begin{lem}\label{claim.A}
For each $X\in\mathcal{R}_{\al}$,
there is a $Y\le_{\al} X$ such that for each $a,b\le_{\mathrm{fin}}^{\al} Y$ in $\hat{\mathcal{F}}$,
$Y/(a,b)$ decides for $a$ and $b$.
\end{lem}

\begin{defn}\label{defn.irred}
Let $\mathcal{F}$ be a front on $[\emptyset, X]$ for some $X\in\mathcal{R}_{\al}$, and let $\vp$ be a function on $\mathcal{F}$.
\begin{enumerate}
\item
$\vp$ is {\em inner} if $\vp(a)\sse a$ for all $a\in\mathcal{F}$.
\item
$\vp$ is {\em Nash-Williams} if $\vp(a)\not\sqsubseteq \vp(b)$, for all $a\ne b\in\mathcal{F}$.
\item 
$\vp$ is {\em Sperner} if $\vp(a)\not\sse \vp(b)$ for all $a\ne b\in\mathcal{F}$ 
\end{enumerate}
\end{defn}

\begin{defn}\label{defn.canonical}
Let $X\in\mathcal{R}_{\al}$, $\mathcal{F}$ be a front on $[\emptyset,X]$, and  $\R$ an equivalence relation on $\mathcal{F}$.
We say that $\R$ is {\em canonical} if and only if there is an
 inner  Nash-Williams  function 
$\vp$ on $\mathcal{F}$ such that
\begin{enumerate}
\item
for all $a,b\in\mathcal{F}$, $a\R b$ if and only if $\vp(a)=\vp(b)$; and
\item
$\vp$ is maximal among all inner Nash-Williams functions satisfying (1).
That is,
for any other inner Nash-Williams function $\vp'$ on $\mathcal{F}$ satisfying (1), there is a $Y\le_{\al} X$ such that $\vp'(a)\sse\vp(a)$ for all $a\in\mathcal{F}|Y$.
\end{enumerate}
\end{defn}

\begin{rem}\label{rem.canonical}
As in \cite{Dobrinen/Todorcevic11},
the map $\vp$ constructed in the 
 proof of Theorem \ref{thm.PRR_al} will in fact be Sperner. 
Moreover, this $\vp$ is also the only such inner  Nash-Williams map  with the additional property $(*)$ that there is a $Z\le_{\al} C$ such that for each $s\in\mathcal{F}|Z$ there is a $t\in\mathcal{F}$ such that $\vp(s)=\vp(t)=s\cap t$.
\end{rem}

The following is part of the general induction scheme discussed in Section \ref{sec.tRs}.
\vskip.1in

\noindent\underbar{Induction Hypothesis}.
Suppose that $2\le \al<\om_1$;
for all $1\le\beta<\al$,
Theorems \ref{thm.PRR_al}, \ref{thm.FCTR^al_n}  and 
\ref{thm.FCTR^beta_n} and Lemma \ref{lem.n_0n_1} (in that order)
 hold for $\mathcal{R}_{\beta}$;
and 
Theorems 
\ref{thm.FCT_R_al(n)}, \ref{thm.PigeonR_al(n)}, \ref{thm.R_altRs}, and \ref{thm.canon.eq.rel.R(n)} 
 hold for $\mathcal{R}_{\al}$.
\vskip.1in

Recall Remark
\ref{rem.barrier}, that for any front $\mathcal{F}$ on some $X\in\mathcal{R}_{\al}$, there is a $Y\le_{\al} X$ such that $\mathcal{F}|Y$ is a barrier.
Thus, the following main theorem yields the analogue of the Pudlak-\Rodl\ Theorem.

\begin{thm}\label{thm.PRR_al}
Suppose $A\in\mathcal{R}_{\al}$,  $\mathcal{F}$ is  a front on $[\emptyset,A]$ and  $\R$ is an equivalence relation on $\mathcal{F}$.
Then
 there is a $C\le_{\al} A$ such that $\R$ is canonical on 
$\mathcal{F}| C$.
\end{thm}

\begin{proof}
Let $A\in\mathcal{R}_{\al}$, 
let $\mathcal{F}$ be a given front on $[\emptyset,A]$,
and let $\R$ be an equivalence relation on $\mathcal{F}$.
Let   $f:\mathcal{F}\ra\bN$ be any mapping which induces  $\R$. 
By thinning if necessary, we may assume that 
 $A$ satisfies Lemma \ref{claim.A}.
Let $(\hat{\mathcal{F}}\setminus\mathcal{F})|X$ denote the collection of those $a\in\hat{\mathcal{F}}\setminus\mathcal{F}$ such that $a\le^{\al}_{\mathrm{fin}} X$.

\begin{claim}\label{claim.E}
There is a $B\le_{\al} A$ such that for  all $a\in(\hat{\mathcal{F}}\setminus\mathcal{F})| B$,
letting $n=|a|$,
there is an 
$S_a\in\mathfrak{S}_{\al}(n)$ such that,
letting $\E_a$ denote $\E_{S_a}$,
for all $u,v\in  \mathcal{R}_{\al}(n)|B/a$, 
$B$ mixes $a\cup u$ and $a\cup v$ 
if and only if $u \E_a v$.
\end{claim}

\begin{proof}
For any $Z\le_{\al} A$ and 
$a\in\mathcal{AR}^{\al}|A$,
let $P(a,Z)$ denote the following statement:
``If
 $a\in\hat{\mathcal{F}}\setminus\mathcal{F}$,
then
there is an
$S_a\in\mathfrak{S}_{\al}(|a|)$
 such that for all $u,v\in
\mathcal{R}_{\al}(|a|)|Z/a$,
$Z$ mixes $a\cup u$ and $a\cup v$ if and only if $u\E_{S_a} v$.''
We shall show that for each $X\le_{\al} A$ and $a\in\mathcal{AR}^{\al}|A$, there is a $Z\le_{\al} X$ for which  $P(a,Z)$ holds.
The claim  then follows from Lemma \ref{lem.1}.

Let $X\le_{\al} A$ and  $a\in\hat{\mathcal{F}}\setminus\mathcal{F}$ be given, and let $n=|a|$.
Let $\E$ denote the following equivalence relation of mixing  on 
$\mathcal{R}_{\al}(n)|A/a$: For all $u,v\in\mathcal{R}_{\al}(n)|A/a$,
\begin{equation}
 u\E v\ \Leftrightarrow\ A\mathrm{\ mixes\ }a\cup u\mathrm{\ and \ } a\cup v.
\end{equation}
By Theorem \ref{thm.canon.eq.rel.R(n)}, there is an $S\in\mathfrak{S}_{\al}(n)$ and a $Y\le_{\al} X$ such that $\E\re(\mathcal{R}_{\al}(n)|Y)$ is given by $\E_S$.
Take a $Z\le_{\al}Y/a$
and let $S_a$ denote this $S$.
Then $P(a,Z)$ holds.
\end{proof}

Fix $B$ be as in Claim \ref{claim.E}.
For $a\in(\hat{\mathcal{F}}\setminus\mathcal{F})|B$ and
 $n=|a|$,
 let $S_a$ denote the member of $\mathfrak{S}_{\al}(n)$ such that $\E_a=\E_{S_a}$, and 
let $\E_a$ denote $\E_{S_a}$.
We say that $a$ is  
 {\em $\E_a$-mixed by $B$}, meaning that  
 for all $u,v\in
\mathcal{R}_{\al}(n)|B/a$, 
$B$ mixes $a\cup u$ and $a\cup v$ if and only if $u\E_a  v$.

\begin{defn}\label{def.vp^1}
For $a\in\hat{\mathcal{F}}|B$, $n=|a|$,
and $i< n$, 
define 
\begin{equation}
\vp_{r^{\al}_i(a)}(a(i))
=\pi_{S_{r^{\al}_i(a)}}(a(i)).
\end{equation}
For $a\in\mathcal{F}|B$, define
\begin{equation}
\vp(a)=\bigcup_{i<|a|}\vp_{r^{\al}_i(a)}(a(i)).
\end{equation}
\end{defn}

The proof of the following claim is exactly the same as the one given for Claim 4.17 in \cite{Dobrinen/Todorcevic11}.

\begin{claim}\label{claim.G}
The following are true for all $X\le_{\al} B$ and all $a,b\in\hat{\mathcal{F}}| B$.
\begin{enumerate}
\item[(A1)]
Suppose $a\not\in\mathcal{F}$ and $n=|a|$.
Then $X$ mixes $a\cup u$ and $t$ for at most one $\E_a$ equivalence class of $u$'s in
$\mathcal{R}_{\al}(n)|B/a$.
\item[(A2)]
If $X/(a,b)$ separates $a$ and $b$,
then $X/(a,b)$ separates $a\cup x$ and $b\cup y$ for all  $x,y\in\Ext(X/(a,b))$ such that $a\cup x,b\cup y\in\hat{\mathcal{F}}$.
\item[(A3)]
Suppose $a\not\in\mathcal{F}$ and  $n=|a|$.
Then 
$S_a=\{\emptyset\}$
if and only if $X$ mixes $a$ and $a\cup u$
for all $u\in\mathcal{R}_{\al}(n)|B/a$.
\item[(A4)]
If $a\sqsubset b$
and  $\vp(a)=\vp(b)$,
then $X$ mixes $a$ and $b$.
\end{enumerate}
\end{claim}

\begin{claim}\label{claim.Tsametype}
If $a,b\in(\hat{\mathcal{F}}\setminus\mathcal{F})|B$ are mixed by $B$, then $S_a$ and $S_b$ are isomorphic.
Moreover, there is a $C\le_{\al} B$ such that for all $a,b\in(\hat{\mathcal{F}}\setminus\mathcal{F})|C$, for all  $u\in\mathcal{R}_{\al}(|a|)|C/(a,b)$ and $v\in\mathcal{R}_{\al}(|b|)|C/(a,b)$,
 $C$ mixes $a\cup u$ and $b\cup v$
 if and only if $\vp_a(u)=\vp_b(v)$.
\end{claim}

\begin{proof}
Suppose $a,b\in (\hat{\mathcal{F}}\setminus\mathcal{F})|B$ are mixed by $B/(a,b)$, and let $X\le_{\al} B$. 
By possibly thinning $X$, we may assume that $X\le_{\al} B/(a,b)$.
Let $i=|a|$ and $j=|b|$.

Suppose that $S_a=\{\emptyset\}$ and $S_b\ne \{\emptyset\}$.
By  (A1), $B/(a,b)$ mixes $a$ and $b\cup v$  for at most one $\E_a$ equivalence class of $v$'s in $\mathcal{R}_{\al}(j)|B/b$.  
Since $S_b\ne \{\emptyset\}$, there is a $Y\le_{\al} X/(a,b)$ such that for each $v\in\mathcal{R}_{\al}(j)|Y$,
$Y$ separates $a$ and $b\cup v$.
Since $S_a=\{\emptyset\}$,
it follows from (A4) that
for all $u\in\mathcal{R}_{\al}(i)|Y$, $Y$ mixes $a$ and  $a\cup u$.
If there are $u\in\mathcal{R}_{\al}(i)|Y$
and $v\in\mathcal{R}_{\al}(j)|Y$ such that $Y$ mixes $a\cup u$ and $b\cup v$,
then $Y$ mixes $a$ and $b\cup v$, by transitivity of mixing.
This contradicts that for each $v\in\mathcal{R}_{\al}(j)|Y$,
$Y$ separates $a$ and $b\cup v$.
Therefore, all extensions of $a$ and $b$ into $Y$ are separated.
But then $a$ and $b$ are separated, contradiction.
Hence, $S_b$ must also be $\{\emptyset\}$.
By a similar argument, we conclude that $S_a=\{\emptyset\}$ if and only if $S_b=\{\emptyset\}$.
Hence,
 $\vp_a(u)=\vp_b(v)=\{\emptyset\}$ for all 
$u\in\mathcal{R}_{\al}(i)|B$ and $v\in\mathcal{R}_{\al}(j)|B$.

Suppose now that both  $S_a$ and $S_b$ are not $\{\emptyset\}$.
Let $X\le_{\al} B/(a,b)$ and $m=\max(i,j)+1$.
Let
\begin{align}
\mathcal{Z}_{a}&
=\{Y\le_{\al} X:B/(a,b)\mathrm{\ separates\ } a\cup Y(i)\mathrm{\ and\ }b\cup \pi_{j,m}^{\al}(Y(m))\}\cr
\mathcal{Z}_{b}&
=\{Y\le_{\al} X:B/(a,b)\mathrm{\ separates\ } a\cup \pi_{i,m}^{\al}(Y(m))\mathrm{\ and\ }b\cup Y(j)\}. 
\end{align}
Applying the Abstract Ellentuck Theorem to the sets $\mathcal{Z}_a$ and $\mathcal{Z}_b$, 
we  obtain an $X'\le_{\al} X$ such that $[0,X']\sse\mathcal{Z}_a\cap\mathcal{Z}_b$,
since both 
$S_a$ and $S_b$ are not $\{\emptyset\}$.
Thus,
 for all $u\in\mathcal{R}_{\al}(i)|X'$ and $v\in\mathcal{R}_{\al}(j)|X'$,
   $a\cup u$ and $b\cup v$ may be mixed by $B/(a,b)$ only if $u$ and $v$ are subtrees of the same $X'(l)$ for some $l$.

For $l\in\{i,j\}$ and $k\ge m$,
let $\mathfrak{I}_{\al}(l,k)$ denote the collection of all $S\sse\bS_{\al}(k)$ such that $S\cong\bS_{\al}(l)$.
So $\mathfrak{I}_{\al}(l,k)$ consists of exactly those $S\in\mathfrak{S}_{\al}(k)$ such that $\pi_S:\mathcal{R}_{\al}(k)\ra\mathcal{R}_{\al}(l)$.
Note that each $\mathfrak{I}_{\al}(l,k)$ is finite.
For  each pair $S\in\mathfrak{I}_{\al}(i,k)$,
$S'\in\mathfrak{I}_{\al}(j,k)$,
let
\begin{equation}
\mathcal{X}_{S,S'}=\{Y\le_{\al} X': B\mathrm{\ mixes\ } a\cup\pi_{S}(Y(k))\mathrm{\ and\ } b\cup\pi_{S'}(Y(k))\}.
\end{equation}

Diagonalize over $k\ge m$ as follows.
Let $Y_m=X'$.
Given $Y_k$, 
apply the Abstract Ellentuck Theorem to $\mathcal{X}_{S,S'}$ for all pairs
$S,S'$ from 
 $\mathfrak{I}_{\al}(i,k)$,
$S'\in\mathfrak{I}_{\al}(j,k)$, respectively, to
 obtain a $Y_{k+1}\le_{\al} Y_k$  which is homogeneous for  $\mathcal{X}_{S,S'}$, for each  such pair.
Define 
\begin{equation}
Y=r^{\al}_{m}(Y_m)\cup \bigcup\{Y_{k+1}(k): k\ge m\}.
\end{equation}
Then $Y$ is homogeneous for $\mathcal{X}_{S,S'}$  for all $k\ge m$ and  all pairs
 $S\in\mathfrak{I}_{\al}(i,k)$,
$S'\in\mathfrak{I}_{\al}(j,k)$.

\begin{subclaimn}
There is a $Z\le_{\al} Y$ such that for each $k\ge m$, each pair 
 $S\in\mathfrak{I}_{\al}(i,k)$,
$S'\in\mathfrak{I}_{\al}(j,k)$,
 and each $Z'\le_{\al} Z$,
if $\vp_{s}(\pi_{S}(Z'(k)))\ne\vp_{t}(\pi_{S'}(Z'(k)))$,
then
$[\emptyset,Z']\cap\mathcal{X}_{S,S'}=\emptyset$.
\end{subclaimn}

Suppose not.  
Then in particular for $Y$,
there are $k$, $S\in\mathfrak{I}_{\al}(i,k)$,
$S'\in\mathfrak{I}_{\al}(j,k)$, and $Z\le_{\al} Y$ such that 
$\vp_{s}(\pi_{S}(Z(k)))\ne\vp_{t}(\pi_{S'}(Z(k)))$,
but $[0,Z]\cap\mathcal{X}_{S,S'}\ne\emptyset$.
Since $Y$ is already homogeneous for $\mathcal{X}_{S,S'}$, 
it must be the case that $[0,Y]\sse\mathcal{X}_{S,S'}$;
hence, $[0,Z]\sse\mathcal{X}_{S,S'}$.
Furthermore,
$\vp_{s}(\pi_{S}(Z(k)))\ne\vp_{t}(\pi_{S'}(Z(k)))$  implies that 
$\vp_{s}(\pi_{S}(Z'(k)))\ne\vp_{t}(\pi_{S'}(Z'(k)))$ for all $Z'\le_{\al} Z$,
since $\vp_s,\pi_S,\vp_t$, and $\pi_{S'}$ are projection maps.

We claim that $\pi_{S_a}(S)=\pi_{S_b}(S')$.
Suppose  there is some 
$s\in 
\pi_{S_a}(S)\setminus\pi_{S_b}(S')$.
Take $w,w'\in \mathcal{R}_{\al}(k)|Z(k')$  for some $k'$ large enough
 such that 
$w$ and $w'$ differ exactly on their  elements in the place $s$ and all extensions of $s$.
Let $u=\pi_{S_a}\circ\pi_{S}(w)$, $u'=\pi_{S_a}\circ\pi_{S}(w')$, 
$v=\pi_{S_b}\circ\pi_{S'}(w)$,
and $v'=\pi_{S_b}\circ\pi_{S'}(w')$.
Then 
$u\not{\E}_a\ u'$ but
$v\E_b v'$.
Since $[\emptyset,Z]\sse\mathcal{X}_{S,S'}$, 
$B/(a,b)$ mixes $a\cup u$ and $b\cup v$, and $B/(a,b)$ mixes $a\cup u'$ and $b\cup v'$.
$B/(a,b)$ mixes $b\cup v$ and $b\cup v'$, since $v\E_b v'$.
Hence, by transitivity of mixing, $B/(a,b)$ mixes $a\cup u$ and $a\cup u'$, contradicting that $u\not\E_a  u'$.
Likewise, we obtain a contradiction if 
there is some 
$s\in 
\pi_{S_b}(S')\setminus
\pi_{S_a}(S)$.
Therefore, the Subclaim holds.
\vskip.1in

By the Subclaim,
  the following holds.
There is a $Z\le_{\al} Y$ such that for all $u\in\mathcal{R}_{\al}(i)|Z$ and $v\in\mathcal{R}_{\al}(j)|Z$,
 if $a\cup u$ and $b\cup v$ are mixed by $B/(a,b)$, then $\vp_{a}(u)=\vp_b(v)$.
It follows that $S_a$ and $S_b$ must be isomorphic.
Thus, we have shown that there is a $Z\le_{\al} X$ such that for all $u\in\mathcal{R}_{\al}(i)|Z$ and $v\in\mathcal{R}_{\al}(j)|Z$,
if $B/(a,b)$ mixes $a\cup u$ and $b\cup v$, then $\vp_a(u)=\vp_b(v)$.

It remains to show that for all $u\in\mathcal{R}_{\al}(i)|Z$ and $v\in\mathcal{R}_{\al}(j)|Z$,
 if $\vp_a(u)=\vp_b(v)$, then $Z$ mixes $a\cup u$ and $b\cup v$.
Let $k\ge m$ and let
$S\in\mathfrak{I}_{\al}(i,k)$,
$S'\in\mathfrak{I}_{\al}(j,k)$,
 be any pair 
such that for all $w\in\mathcal{R}_{\al}(k)|Z$,
$\vp_a(\pi_{S}(w))=\vp_b(\pi_{S'}(w))$.
We will show that $[\emptyset, Z]\sse\mathcal{X}_{S,S'}$.

Assume towards a contradiction that $[\emptyset,Z]\cap\mathcal{X}_{S,S'}=\emptyset$.
Then for all $w\in \mathcal{R}_{\al}(k)|Z$,
$Z$ separates $a\cup \pi_{S}(w)$ and $b\cup \pi_{S'}(w)$.
First, let 
$T\in\mathfrak{I}_{\al}(i,k)$,
$T'\in\mathfrak{I}_{\al}(j,k)$,
 be any  pair such that $\vp_a(\pi_{T}(x))=\vp_b(\pi_{T'}(x))$ for any (all) $x\in \mathcal{R}_{\al}(k)|Z$.
Then there are $x,y\in \mathcal{R}_{\al}(k)|Z$ such that $\pi_{S}(x)\E_a \pi_{T}(y)$ and $\pi_{S'}(x)\E_b \pi_{T'}(y)$.
 $Z$ mixes $a\cup \pi_{S}(x)$ and $a\cup \pi_{T}(y)$, and $Z$ mixes $b\cup \pi_{S'}(x)$ and $b\cup \pi_{T'}(y)$.
Thus, $Z$ must separate $a\cup \pi_{T}(w)$ and $b\cup \pi_{T'}(w)$ for all $w\in \mathcal{R}_{\al}(k)|Z$.
Second, let  $T,T'$  be any pair such that
 $\vp_a(\pi_{T}(x))\not=\vp_b(\pi_{T'}(x))$.
 Then
$Z$ separates $a\cup \pi_{T}(x)$ and $b\cup \pi_{T'}(x)$.
Thinning, we obtain a   $Z'\le_{\al} Z/r^{\al}_k(Z)$ which  separates $a$ and $b$,
 contradiction.
 Therefore, $[\emptyset,Z]\sse\mathcal{X}_{S,S'}$,
 and thus $Z$ mixes $a\cup \pi_{S}(W(k))$ and $b\cup \pi_{S'}(W(k))$ for all $W\le_{\al} Z$.

Hence, for all such  pairs $S,S'$, we have that
$\vp_a(\pi_{S}(w))=\vp_b(\pi_{S'}(w))$ if and only if $[\emptyset,Z]\sse\mathcal{X}_{S,S'}$.
Thus,  for all $u\in\mathcal{R}_{\al}|Z$ and $v\in\mathcal{R}_{\al}|Z$,
$Z$ mixes $a\cup u$ and $b\cup v$ if and only if 
$\vp_a(u)=\vp_b(v)$.

Finally, we have shown that for all $a,b\in(\hat{\mathcal{F}}\setminus\mathcal{F})|B$ and each $X\le_{\al} B$, there is a $Z\le_{\al} X$ such that for all $u\in\mathcal{R}_{\al}(i)|Z$ and $v\in\mathcal{R}_{\al}(j)|Z$,
$Z$ mixes $a\cup u$ and $b\cup v$ if and only if $\vp_a(u)=\vp_b(v)$.
By Lemma \ref{lem.1},
there is a $C\le_{\al} B$ for which the Claim holds.
\end{proof}

The proofs of the next three claims are the same as the proofs of Claims 4.19, 4.20 and 4.21 in  \cite{Dobrinen/Todorcevic11}.

\begin{claim}\label{claim.phisamefsame}
For all $a,b\in\hat{\mathcal{F}}|C$,
if $\vp(a)=\vp(b)$, then $a$ and $b$ are mixed by $C$.
Hence, for all $a,b\in\mathcal{F}|C$,
if $\vp(a)=\vp(b)$, then $f(a)=f(b)$.
\end{claim}

\begin{claim}\label{claim.front}
For all $a,b\in\mathcal{F}|C$, $\vp(a)\not\sqsubset \vp(b)$.
\end{claim}

\begin{claim}\label{claim.L}
For all $a,b\in\mathcal{F}|C$, if $f(a)=f(b)$,
then $\vp(a)=\vp(b)$.
\end{claim}

It remains to show that $\vp$ witnesses that $\R$ is canonical.
By definition, $\vp$ is inner, and by Claim \ref{claim.front},
$\vp$ is Nash-Williams.
By Claims \ref{claim.phisamefsame} and  \ref{claim.L}, we have that
for each $a,b\in\mathcal{F}|C$, $a \R b$ if and only if $\vp(a)=\vp(b)$.
Thus, it only remains to show that $\vp$ is maximal among all inner Nash-Williams maps $\vp'$ on $\mathcal{F}|C$ which also represent the equivalence relation $\R$.
Toward this end, we prove the following Lemma.

\begin{lem}\label{lem.irredssephi}
Suppose $X\le_{\al} C$ and $\vp'$ is an inner  function on $\mathcal{F}|X$ which represents $\R$.
Then 
there is a $Y\le_{\al} X$ such that for each $a\in\mathcal{F}|Y$, 
for each $i<|a|$, there is an 
$S'_{r_i^{\al}(a)}\in\mathfrak{S}_{\al}(i)$ 
such that
$S'_{r_i^{\al}(a)}\sse S_{r_i^{\al}(a)}$ 
and 
the following hold.
\begin{enumerate}
\item
For each $b\in\mathcal{F}|Y$ for which $b\sqsupset r_i^{\al}(a)$,
$\vp'(b)\cap b(i)=\pi_{S'_{r_i^{\al}(a)}}(b(i))$.
\item
$\vp'(a)=\bigcup\{\pi_{S'_{r_i^{\al}(a)}}(a(i)):i<|a|\}\sse \vp(a)$.
\end{enumerate}
\end{lem}

\begin{proof}
Let $X\le_{\al} C$ and $\vp'$ 
satisfy the hypotheses.
Fix any $a\in(\hat{\mathcal{F}}\setminus\mathcal{F})|C$, $i<|a|$, and $X'\le_{\al} X/a$.
For each $k\ge i$ and $S\in\mathfrak{S}_{\al}(i,k)$, let $\mathcal{X}_S=\{Y\le_{\al} X': \vp'(a\cup Y[i,j))\cap Y(i)=\pi_S(Y(i))\}$,  where $j$ is such that $a\cup Y[i,j)\in\mathcal{F}$.
Since $\vp'$ is inner, 
following the argument in Lemma \ref{lem.cohere},
we construct
 an $X''\le_{\al} X'$ such that the following holds:
There is an $S'_{r^{\al}_i(a)}\in\mathfrak{S}_{\al}(i)$ 
such that for each $b\in \mathcal{F}$ extending $r^{\al}_i(a)$ with $b/ r^{\al}_i(a)\in\Ext(X'')$,
$\vp'(b)\cap b(i)=\pi_{S'_{r_i^{\al}(a)}}(b(i))$.
By Lemma \ref{lem.1},
there is a $Y\le_{\al} X$ such that for each $a\in\mathcal{F}|Y$ and each $i<|a|$,
there is an $S'_{r_i^{\al}(a)}$ satisfying (1).
Thus,  for each $a\in\mathcal{F}|Y$,
\begin{equation}
\vp'(t)=\bigcup\{\pi_{S'_{r_i^{\al}(a)}}(a(i)):i<|a|\}.
\end{equation}

Note that each $S'_{r_i^{\al}(a)}$ must be contained within the $S_{r_i^{\al}(a)}$ for the $\vp$ already attained associated with $\E_{r^{\al}_i(a)}$-mixing of immediate extensions of $r^{\al}_i(a)$.
Otherwise, there would be $u,v\in\mathcal{R}_{\al}(i)|Y/r^{\al}_i(a)$ such that $r^{\al}_i(a)\cup u$ and $r^{\al}_i(a)\cup v$ are mixed, yet all extensions of them have different $\vp'$ values, which would contradict that $\vp'$ induces the same equivalence relation as $f$.
Thus, for each $a\in\mathcal{F}|Y$,
$\vp'(a)\sse\vp(a)$.
\end{proof}

By
Lemma \ref{lem.irredssephi}, $\R$ is canonical on $\mathcal{F}|C$, which  finishes the proof of the theorem.
\end{proof}

\begin{rem}
The map $\vp$ from Theorem \ref{thm.PRR_al} has the following property.
One can thin to a $Z$  such that 
\begin{enumerate}
\item[$(*)$]
for each 
$s\in\mathcal{F}|Z$, there is a $t\in\mathcal{F}$ such that $\vp(s)=\vp(t)=s\cap t$.
\end{enumerate}
This is not the case for any smaller inner map $\vp'$, by Lemma \ref{lem.irredssephi}.
For  suppose $\vp'$ is an inner map representing $\R$,  $\vp'$ satisfies  the conclusions of Lemma \ref{lem.irredssephi} on $\mathcal{F}|Y$, and there is an $s\in\mathcal{F}|Y$ for which $\vp'(s)\subsetneq \vp(s)$.
Then there is some $i<|s|$ for which  $S'_{r^{\al}_i(s)}\subsetneq S_{r^{\al}_i(s)}$.
This implies that $\vp'(t)\subsetneq \vp(t)$ for every $t\in\mathcal{F}|Y$ such that $t\sqsupset r_i(s)$.
Recall that $\vp'(t)=\vp'(s)$ if and only if $\vp(t)=\vp(s)$; and in this case, $\vp(t)\cap\vp(s)\sse t\cap s$.
It follows that for any $t$ for which $\vp'(t)=\vp'(s)$, 
$\vp'(t)\cap\vp'(s)$ will always be a proper subset of $t\cap s$.
Thus, $\vp$ is the minimal inner map for which property $(*)$ holds.
\end{rem}

As shown in \cite{Dobrinen/Todorcevic11} for $\mathcal{R}_1$, 
this is the best possible:  there are examples of fronts on which there are inner maps $\vp'$ such that $\vp'(a)\subsetneq \vp(a)$ for all $a\in\mathcal{F}|C$.

Recall Definition \ref{defn.canon-chunk}.
For $n<\om$ and $X\in\mathcal{R}_{\al}$,
an equivalence relation $\R$ on the front $\mathcal{AR}^{\al}_n|X$ is {\em canonical} iff for each $i<n$ there is an $S(i)\in\mathfrak{S}_{\al}(i)$ such that 
\begin{equation}
\forall a,b\in\mathcal{AR}^{\al}_n|X,\ 
a\R b \Leftrightarrow \forall i<n\, (\pi_{S(i)}(a(i))=\pi_{S(i)}(b(i))).
\end{equation}
Note that if $n=0$, then $\mathcal{AR}^{\al}_0=\{\emptyset\}$, and every equivalence relation on $\{\emptyset\}$ is trivially canonical.

\begin{thm}[Canonization Theorem for $\mathcal{AR}^{\al}_n$]\label{thm.FCTR^al_n}
Let $1\le n<\om$.
Given any $A\in\mathcal{R}_{\al}$ and any equivalence relation $\R$ on  $\mathcal{AR}^{\al}_n|A$,
there is a $D\le_{\al} A$ such that $\R$ is canonical on $\mathcal{AR}^{\al}_n|D$.
\end{thm}

\begin{proof} 
Let $C\le_{\al} A$ be obtained from Theorem \ref{thm.PRR_al}.
Then for each $a\in\mathcal{AR}_n^{\al}|C$, there is a sequence $\lgl S_{r^{\al}_i(a)}:i<n\rgl$, where each $S_{r^{\al}_i(a)}\in\mathfrak{S}_{\al}(i)$,
satisfying the following:
For all $a,b\in\mathcal{AR}_n^{\al}|C$,
\begin{equation}
a \R b \Leftrightarrow\bigcup_{i<n}\pi_{S_{r^{\al}_i(a) }}(a(i))=\bigcup_{i<n}\pi_{S_{r^{\al}_i(b)}}(b(i)).
\end{equation}
We shall  obtain a $D\le_{\al} C$ such that for all $a,b\in\mathcal{AR}_n^{\al}|D$ and all $i<n$, $S_{r^{\al}_i(a)}=S_{r^{\al}_i(b)}$.

By  the proof of Theorem \ref{thm.PRR_al},
for all $a,b\in\mathcal{AR}_n^{\al}|C$, $S_{r^{\al}_0(a)}=S_{r^{\al}_0(b)}$. 
Let $X_0=C$ and $S(0)=S_{r^{\al}_0(a)}$ for any (all) $a\in\mathcal{AR}_n^{\al}|C$. 
Suppose $j\le n-1$ and  for all $i<j$, $X_i$, and $S(i)$ such that 
$[\emptyset,X_{i}]\sse\mathcal{X}_{S(i)}$,
 where $\mathcal{X}_{S(i)}=\{X\le_{\al} C: S_{r^{\al}_{i}(X)}=S(i)\}$. 
For each $k\ge j$ and each $S\in
\mathfrak{S}_{\al}(j,k)$, define 
\begin{equation}
\mathcal{X}_{S}(j,k)=\{X\le_{\al} C: \pi_{S_{r^{\al}_{j}}(X)}\re\mathcal{R}_{\al}(j)|X(k)=\pi_S\re\mathcal{R}_{\al}(j)|X(k)\}. 
\end{equation} 
These finitely many open sets, $\mathcal{X}_{S}(j,k)$, $S\in\mathfrak{S}_{\al}(j,k)$, cover $[\emptyset,C]$.
Diagonalizing over all $k\ge j$ as in the proof of Lemma \ref{lem.cohere},
 there is some $S(j)\in\mathfrak{S}_{\al}(j)$ and some  $X_{j}\le_{\al} X_{j-1}$ such that $[\emptyset,X_{j}]\sse\mathcal{X}_{S(j)}$,
 where $\mathcal{X}_{S(j)}=\{X\le_{\al} C: S_{r^{\al}_{j}(X)}=S(j)\}$.

Let $D=X_{n-1}$.
Then for all $a,b\in\mathcal{AR}_n^{\al}|D$,
\begin{align}
a \R b &\Leftrightarrow \vp(a)=\vp(b)\cr
&\Leftrightarrow
\forall i<n,\  
\pi_{S_{r^{\al}_i(a)}}(a(i))=\pi_{S_{r^{\al}_i(b)}}(b(i))\cr
&\Leftrightarrow
\forall i<n,\ 
\pi_{S(i)}(a(i))=\pi_{S(i)}(b(i))\cr
&\Leftrightarrow
\forall i<n,\ 
a(i)\E_{S(i)} b(i).
\end{align}
Thus, the equivalence relation $\R$ is canonical on $\mathcal{AR}_n^{\al}|D$.
\end{proof}


\section{The Tukey ordering below $\mathcal{U}_{\al}$ in terms of the Rudin-Keisler ordering}\label{sec.R_alTukey}

In this section, for each $\al<\om_1$, we classify the Rudin-Keisler classes within the Tukey type of any ultrafilter Tukey reducible to $\mathcal{U}_{\al}$, the ultrafilter corresponding to the space $\mathcal{R}_{\al}$.
As a corollary, we obtain the structure of the Tukey types of all ultrafilters Tukey reducible to $\mathcal{U}_{\al}$.

Recall that every topological Ramsey space has its own notion of  Ramsey and selective ultrafilters (see \cite{MR2330595}).
Recall the following definitions  from \cite{Dobrinen/Todorcevic11}.

\begin{defn}[\cite{Dobrinen/Todorcevic11}, \cite{MR2330595}]\label{defn.RamseyufU1}
Let $(\mathcal{R},\le,r)$ be any topological Ramsey space.
\begin{enumerate}
\item
We say that a subset
$\mathcal{C}\sse\mathcal{R}$ {\em  satisfies the Abstract Nash-Williams Theorem} if and only if for each family $\mathcal{G}\sse\mathcal{AR}$ and partition $\mathcal{G}=\mathcal{G}_0\cup\mathcal{G}_1$,
there is a $C\in\mathcal{C}$ and an $i\in 2$ such that  $\mathcal{G}_i|C=\emptyset$.
\item
We say that an ultrafilter $\mathcal{U}$  is  {\em Ramsey for $\mathcal{R}$}
if and only if $\mathcal{U}$ is generated by a subset  $\mathcal{C}\sse\mathcal{R}$ 
which satisfies the Abstract Nash-Williams Theorem.
\item
An ultrafilter generated by a set $\mathcal{C}\sse\mathcal{R}$ is {\em selective for $\mathcal{R}$} if and only if
for each decreasing sequence $X_0\ge X_1\ge\dots$ of members of $\mathcal{C}$, there is another $X\in\mathcal{C}$ such that for each $n<\om$, $X\le X_n/r_n(X_n)$.
\item
We say that an ultrafilter $\mathcal{U}$ is {\em canonical for fronts on $\mathcal{R}$} if and only if
for any front $\mathcal{F}$ on $\mathcal{R}$ and any equivalence relation $\R$ on $\mathcal{F}$,
there is a $U\in\mathcal{U}\cap\mathcal{R}$ such that 
$\R$ is canonical on $\mathcal{F}|U$.
\end{enumerate}
\end{defn}


We fix the following notation for the rest of this section.

\begin{notn}\label{notn.ufonafront}
For each $\al<\om_1$,
let  $\mathcal{U}_{\al}$ denote any ultrafilter on base set $\mathbb{T}_{\al}$ which is Ramsey for $\mathcal{R}_{\al}$ and
canonical for fronts on $\mathcal{R}_{\al}$.
Let $\mathcal{C}_{\al}$ denote $\mathcal{U}_{\al}\cap\mathcal{R}_{\al}$.
We shall say that   $\mathcal{F}\sse\mathcal{AR}^{\al}$ is a {\em front on a set $\mathcal{C}\sse\mathcal{R}_{\al}$}
if $\mathcal{F}$ is Nash-Williams and for each $X\in\mathcal{C}$, there is an $a\in\mathcal{F}$ such that $a\sqsubset X$.
For any front $\mathcal{F}$ on $\mathcal{C}_{\al}$ and
  any $X\in\mathcal{C}_{\al}$,
recall that
 $\mathcal{F}| X$ denotes
$\{a\in\mathcal{F}:a\le_{\mathrm{fin}}^{\al} X\}$.
Let
\begin{equation}
\mathcal{C}_{\al}\re\mathcal{F}
=\{\mathcal{F}| X:X\in\mathcal{C}_{\al}\}.
\end{equation}
\end{notn}

For each $\al<\om_1$,
ultrafilters $\mathcal{U}_{\al}$, which are Ramsey for $\mathcal{R}_{\al}$ and 
canonical for fronts on $\mathcal{R}_{\al}$
exist, assuming  CH or  MA, or forcing with $(\mathcal{R}_{\al},\le^*_{\al})$.
Since $\mathcal{R}_{\al}$ is isomorphic to a dense subset of Laflamme's forcing $\bP_{\al}$ in \cite{Laflamme89},
any ultrafilter  forced by $(\mathcal{R}_{\al},\le^*_{\al})$ is isomorphic to 
an ultrafilter forced  by $(\bP_{\al},\le_{\bP_{\al}}^*)$.
Note that $\mathcal{C}_{\al}$ is
cofinal in $\mathcal{U}_{\al}$.

\begin{rem}
$\mathcal{U}_{\al}$ is isomorphic to the ultrafilter  on base set $[\bT_{\al}]$ generated by  the collection of $[X]$, $X\in\mathcal{C}_{\al}$, which we denote as $\bar{\mathcal{U}}_{\al}$.
(As usual, $[X]$ denotes the set of cofinal branches through $X$, which in this context is exactly the set of maximal nodes in the tree $X$.)
The injection $g:[\bT_{\al}]\ra\bT_{\al}$, where $g(t)=t$ for each $t\in[\bT_{\al}]$, yields $g(\bar{\mathcal{U}}_{\al})=\mathcal{U}_{\al}$.
Though it would perhaps be more standard to consider $[\bT_{\al}]$ as the base set,
we use $\bT_{\al}$ as the base for $\mathcal{U}_{\al}$, as this simplifies notation:
firstly,   $\mathcal{U}_{\al}$ will be generated by true elements of $\mathcal{R}_{\al}$, and secondly, the projection ultrafilters $\pi_S(\mathcal{U}_{\al})$, $S\sse\bS_{\al}(n)$, are then truly projections of $\mathcal{U}_{\al}$.
\end{rem}

\begin{fact}\label{fact.basic}
For each $\al<\om_1$, the following hold.
\begin{enumerate}
\item
If $\mathcal{B}\sse\mathcal{C}_{\al}$ generates $\mathcal{U}_{\al}$, then for each front $\mathcal{F}$ on $\mathcal{B}$ and each 
$\mathcal{G}\sse\mathcal{F}$,
there is a $U\in\mathcal{B}$ such that either $\mathcal{F}|U\sse\mathcal{G}$, or else $\mathcal{F}|U\cap\mathcal{G}=\emptyset$.
\item
Any ultrafilter Ramsey for $\mathcal{R}_{\al}$ is also  selective for $\mathcal{R}_{\al}$.
\end{enumerate}
\end{fact}

(1) follows immediately from the definition of $\mathcal{U}_{\al}$;
(2) is a consequence of Lemma 3.8 in \cite{MR2330595}.

Given a front $\mathcal{F}$ on $\mathcal{C}_{\al}$,
we let $\mathcal{U}_{\al}\re\mathcal{F}$  denote the ultrafilter on base set $\mathcal{F}$ generated by the sets $\mathcal{F}|X$, $X\in\mathcal{C}_{\al}$.
The proofs of  Facts \ref{fact.ufonafront} and \ref{fact.ufequal}
and Proposition  \ref{prop.W=frontuf}
are the same as the proofs of Facts 5.3 and 5.4 and Proposition 5.5 in  \cite{Dobrinen/Todorcevic11}.

\begin{fact}\label{fact.ufonafront}
Let $\al<\om_1$, $\mathcal{B}$ be any cofinal subset of $\mathcal{C}_{\al}$, and  $\mathcal{F}\sse\mathcal{AR}^{\al}$ be any front on $\mathcal{C}_{\al}$.
Then
$\mathcal{B}\re\mathcal{F}$ generates the ultrafilter  $\mathcal{U}_{\al}\re\mathcal{F}$ on the base set $\mathcal{F}$.
\end{fact}

\begin{fact}\label{fact.ufequal}
Suppose $\mathcal{U}$ and $\mathcal{V}$ are proper ultrafilters on the same countable base set, and for each $V\in\mathcal{V}$ there is a $U\in\mathcal{U}$ such that $U\sse V$.
Then $\mathcal{U}=\mathcal{V}$.
\end{fact}

Recall that by Theorem \ref{thm.5}, every Tukey reduction from a p-point to another ultrafilter is witnessed by a continuous cofinal map.
By arguments from \cite{Dobrinen/Todorcevic11}, the following holds.

\begin{prop}\label{prop.W=frontuf}
Let $\al<\om_1$.
Suppose $\mathcal{V}$ is a nonprincipal ultrafilter (without loss of generality on $\bN$) such that  $\mathcal{V}\le_T\mathcal{U}_{\al}$.
Then there is a front $\mathcal{F}$ on $\mathcal{C}_{\al}$ and a function $f:\mathcal{F}\ra\bN$ such that
$\mathcal{V}=f(\mathcal{U}_{\al}\re\mathcal{F})$.
\end{prop}


We now introduce notation which aids in making clear the classification of  ultrafilters which are Rudin-Keisler or Tukey below $\mathcal{U}_{\al}$.

\begin{notn}\label{notn.Phi_n}
Let $\al<\om_1$. 
\begin{enumerate}
\item
For each $n<\om$,
define $\mathcal{U}_{\al}\re\mathcal{R}_{\al}(n)$ to be the filter on the base $\mathcal{R}_{\al}(n)$ generated by the sets $\mathcal{R}_{\al}(n)|X$, $X\in\mathcal{C}_{\al}$.
\item
For each $n<\om$ and each  $S\in\mathfrak{S}_{\al}(n)$,
define $\mathcal{Y}^{\al}_S$ to be the filter on the base set $B_S:=\{\pi_S(u):u\in\mathcal{R}_{\al}(n)\}$
generated by the collection of sets $\pi_S(\mathcal{R}_{\al}(n)|X):=\{\pi_S(u):u\in\mathcal{R}_{\al}(n)|X\}$, $X\in\mathcal{C}_{\al}$.
\item
For $\beta\le\al$, 
let $\mathcal{Y}^{\al}_{\beta}$ denote $\mathcal{Y}^{\al}_{S_{\beta}}$.
Recall that $S_{\beta}\ (=S^{\al}_{\beta})$ is the downward closed subset of $\bS_{\al}(0)$ of order-type $[\beta,\al+1]$;
that is, $S_{\beta}=\{s\in\bS_{\al}(0):\exists \gamma\in[\beta,\al]\, (\dom(s)=[\gamma,\al])\}\cup\{\emptyset\}$.
\end{enumerate}
\end{notn}

In the next proposition, theorem and corollary,
we highlight the relationships between the various projection ultrafilters of the form $\mathcal{Y}^{\al}_S$, and the ultrafilters of the form $\mathcal{U}_{\al}|\mathcal{R}_{\al}(n)$.

\begin{prop}\label{prop.structureY_S}
Let $\al<\om_1$.
\begin{enumerate}
\item
$\mathcal{U}_{\al}$ is a rapid p-point.
\item
$\mathcal{U}_{\al}\cong\mathcal{U}_{\al}\re\mathcal{R}_{\al}(0)=\mathcal{Y}^{\al}_0$.
\item
$\mathcal{Y}_{\al}^{\al}$ is a Ramsey ultrafilter.
\item
For each $n<\om$ and $S\in\mathfrak{S}_{\al}(n)$,
$\mathcal{Y}_S^{\al}$ is an ultrafilter,  and moreover is a rapid p-point.
\item
Suppose $m\le n$, $S\in\mathfrak{S}_{\al}(m)$, $T\in\mathfrak{S}_{\al}(n)$, and $S\cong T$.
Then $\mathcal{Y}^{\al}_S\cong\mathcal{Y}^{\al}_{T}$.
\item
If  $S\cong \bS_{\al}(k)$, then $\mathcal{Y}^{\al}_S\cong\mathcal{Y}^{\al}_{\bS_{\al}(k)}=\mathcal{U}_{\al}\re\mathcal{R}_{\al}(k)$.
\end{enumerate}
\end{prop}

\begin{proof}
(1) follows from Theorem \ref{thm.Laflammethms}.
To see
(2), recall that $\mathcal{U}_{\al}\cong\bar{\mathcal{U}}_{\al}$.
The map $g:[\bT_{\al}]\ra\mathcal{R}_{\al}(0)$, given by $g(t)=\{s\in\bT_{\al}:s\sqsubseteq t\}$  for each $t\in[\bT_{\al}]$, yields an isomorphism from $\bar{\mathcal{U}}_{\al}$ to $\mathcal{U}_{\al}\re\mathcal{R}_{\al}(0)$.
The equality follows from the fact that
$\pi_{\bS_{\al}(0)}$ is the identity map on $\mathcal{R}_{\al}(0)$.

(3)  follows from the fact that the projection $\pi_{S_{\al}}$ on $\mathcal{R}_{\al}(0)$ yields an isomorphic copy of the Ellentuck space. 
Hence, 
$\mathcal{Y}^{\al}_{\al}$ is Ramsey for the Ellentuck space, which yields that $\mathcal{Y}^{\al}_{\al}$ a Ramsey ultrafilter.

(4)
Let   $S\in\mathfrak{S}_{\al}(n)$.
Let $V$ be any subset of  $B_S$, and
let $\mathcal{H}=\{a\in\mathcal{AR}_{n+1}: \pi_S(a(n))\in V\}$.
Since  $\mathcal{U}_{\al}$ is Ramsey for $\mathcal{R}_{\al}$, there is an $X\in\mathcal{C}_{\al}$ such that either
$\mathcal{AR}_{n+1}^{\al}|X\sse\mathcal{H}$ or else
$\mathcal{AR}_{n+1}^{\al}|X\cap\mathcal{H}=\emptyset$.
In the first case, $V\in\mathcal{Y}_S^{\al}$ and in the second case, $B_S\setminus V\in\mathcal{Y}_S^{\al}$.
Thus, $\mathcal{Y}_S^{\al}$ is an ultrafilter.

Suppose $U_0\contains U_1\contains\dots$ is a decreasing sequence of elements of $\mathcal{Y}_S^{\al}$.
For each $k<\om$, there is some $X_k\in\mathcal{C}_{\al}$ for which $\pi_S(\mathcal{R}_{\al}(n)|X_k)\sse U_k$.
We may take $(X_k)_{k<\om}$ to be a $\le_{\al}$-decreasing sequence.
Since $\mathcal{U}_{\al}$ is selective for $\mathcal{R}_{\al}$, there is an $X\in\mathcal{C}_{\al}$ such that $X/r_k^{\al}(X)\le_{\al} X_k$, for each $k<\om$.
Then $\pi_S(\mathcal{R}_{\al}(n)|X)\sse^*\pi_S(\mathcal{R}_{\al}(n)|X_k)$, for each $k<\om$.
Thus, $\mathcal{Y}_S^{\al}$ is a p-point.

That $\mathcal{Y}_S^{\al}$ is rapid
follows from the fact that $\mathcal{U}_{\al}\re\mathcal{R}_{\al}(n)$ is rapid.
To see this,
 let $h:\om\ra\om$ be a strictly increasing function.
Linearly order $\mathcal{R}_{\al}(n)$ so that all members of $\mathcal{R}_{\al}(n)|\mathbb{T}_{\al}(k)$ appear before all members of $\mathcal{R}_{\al}(n)|\mathbb{T}_{\al}(k+1)$ for all $k\ge n$.
For any tree $u\sse\bT_{\al}$, let $m(u)$ denote the least $l$ such that $\lgl l\rgl\in u$.
For each $X\in\mathcal{C}_{\al}$, there is a $Y\le_{\al} X$ such that
$m(Y(n))>h(1)$, 
$m(Y(n+1))>h(1+|\mathcal{R}_{\al}(n)|\mathbb{T}_{\al}(n+1)|)$,
 and in general, for $k\ge n$,
\begin{equation}
m(Y(k))
>h(\Sigma_{n\le i\le k} |\mathcal{R}_{\al}(n)|\mathbb{T}_{\al}(i)|).
\end{equation}
Since $\mathcal{U}_{\al}$  is selective for $\mathcal{R}_{\al}$, there is a $Y\in\mathcal{C}_{\al}$ with this property, which yields that
$\mathcal{U}_{\al}\re\mathcal{R}_{\al}(n)$ is rapid for the function $h$.
Since for each $u\in\mathcal{R}_{\al}(n)$, $|\pi_S(u)|\le |u|$,
it follows that $\pi_S(\mathcal{R}_{\al}(n)|Y)$ witnesses that $\mathcal{Y}^{\al}_S$ is rapid for the function $h$.
Since $h$ was arbitrary, (4) holds.

(5) 
Suppose that $m\le n$, $S\in\mathfrak{S}_{\al}(m)$, $T\in\mathfrak{S}_{\al}(n)$, and $S\cong T$.
Then  $B_{T}\sse B_S$.
Moreover, there is an $X\in\mathcal{C}_{\al}$ such that 
$B_S|X\sse B_{T}$.
Thus, modulo negligible subsets of the bases, $\mathcal{Y}^{\al}_S$ is actually equal to $\mathcal{Y}^{\al}_{T}$.
The identity map on $B_{T}$ witnesses that $\mathcal{Y}^{\al}_S\le_{RK}\mathcal{Y}^{\al}_{T}$.
Given $X\in\mathcal{C}_{\al}$ such that $B_S|X\sse B_{T}$ and $\bT_{\al}\setminus X$ is infinite,
the identity map on $B_S|X$ witnesses that $\mathcal{Y}^{\al}_{T}\le_{RK}\mathcal{Y}^{\al}_S$.
(6) follows from (5).
\end{proof}

For $S$ and $T$ downward closed subsets of $\bS_{\al}$, we say that $S$ {\em embeds into} $T$, or $S$ {\em is isomorphic to a subset of} $T$, if  there is an injection $\iota:S\ra T$ which preserves lexicographic ordering  (recall Definition \ref{def.lex}) and such that the image $\iota(S)$ is downward closed in $T$.

\begin{thm}\label{prop.RKstructure}
Let $m,n<\om$, and let $S\in\mathfrak{S}_{\al}(m)$ and $T\in\mathfrak{S}_{\al}(n)$.
\begin{enumerate}
\item
If $S$ embeds into $T$, then 
$\mathcal{Y}^{\al}_S\le_{RK}\mathcal{Y}^{\al}_T$.
\item
If $\mathcal{V}\le_{RK}\mathcal{Y}^{\al}_T$,
then $\mathcal{V}\cong\mathcal{Y}^{\al}_{T'}$ for some  $T'\sse T$ such that $T'\in\mathfrak{S}_{\al}(n)$.
\item
If
$\mathcal{Y}^{\al}_S\le_{RK}\mathcal{Y}^{\al}_T$
then
$S$ embeds into $T$.
\item
$\mathcal{Y}^{\al}_S\cong\mathcal{Y}^{\al}_T$ iff
$S\cong T$.
\end{enumerate}
\end{thm}

\begin{proof}
(1)
 Suppose that $S$ is isomorphic to a subset $T'\sse T$.
By Proposition \ref{prop.structureY_S} (5), $\mathcal{Y}^{\al}_S\cong\mathcal{Y}^{\al}_{T'}$.
The projection map $\pi_{T'}$ from $B_T$ to $B_{T'}$ witnesses that $\mathcal{Y}^{\al}_{T'}\le_{RK}\mathcal{Y}^{\al}_T$.

(2)
Suppose $\mathcal{V}\le_{RK}\mathcal{Y}^{\al}_T$, and without loss of generality, assume that $\om$ is the base set for $\mathcal{V}$.
Let $\theta:B_T\ra\om$ witness $\mathcal{V}\le_{RK}\mathcal{Y}^{\al}_T$; so $\theta(\mathcal{Y}^{\al}_T)=\mathcal{V}$.
Let  $f=\theta\circ\pi_T$ so that $f:\mathcal{R}_{\al}(n)\ra\om$.
By the Canonization Theorem \ref{thm.FCTR^al_n}  for $\mathcal{R}_{\al}(n)$
and the definition of $\mathcal{U}_{\al}$,
there is a $C\in\mathcal{C}_{\al}$ and a $T'\in\mathfrak{S}_{\al}(n)$ such that for all $u,v\in\mathcal{R}_{\al}(n)|C$,
$f(u)=f(v)$ iff $\pi_{T'}(u)=\pi_{T'}(v)$.
Thus, there is a bijection between $f''\mathcal{R}_{\al}(n)|C$ and $\pi_{T'}''\mathcal{R}_{\al}(n)|C$.

Suppose that $T'\setminus T\ne\emptyset$.
Then there are $u,v\in\mathcal{R}_{\al}(n)|C$ such that $\pi_T(u)=\pi_T(v)$ but $\pi_{T'}(u)\ne\pi_{T'}(v)$.
$\pi_T(u)=\pi_T(v)$ implies that $\theta(\pi_T(u))=\theta(\pi_T(v))$, which implies that $f(u)=f(v)$.
However, $\pi_{T'}(u)\ne\pi_{T'}(v)$ implies that $f(u)\ne f(v)$, contradiction.
Thus $T'\sse T$.
Hence,
$\mathcal{V}=\theta(\mathcal{Y}^{\al}_T)=\theta(\pi_T(\mathcal{U}_{\al}\re\mathcal{R}_{\al}(n)))
=f(\mathcal{U}_{\al}\re\mathcal{R}_{\al}(n))\cong\pi_S(\mathcal{U}_{\al}\re\mathcal{R}_{\al}(n))=\mathcal{Y}^{\al}_S$.

(3) 
Suppose that $\theta:B_T\ra B_S$ witnesses that $\mathcal{Y}^{\al}_S\le_{RK}\mathcal{Y}^{\al}_T$.
By (2), there is a $T'\sse T$ such that $\mathcal{Y}^{\al}_S\cong\mathcal{Y}^{\al}_{T'}$.
Let $\theta: B_{T'}\ra B_S$ be an isomorphism witnessing this.
Let $f:\mathcal{R}_{\al}(n)\ra B_S$ by letting $f=\theta\circ\pi_{T'}$.
By Theorem \ref{thm.canon.eq.rel.R(n)} and the definition of $\mathcal{U}_{\al}$,
there is some $T''\in\mathfrak{S}_{\al}(n)$ and $U\in\mathcal{C}_{\al}$ such that 
for all $u,v\in\mathcal{R}_{\al}(n)|U$,
$f(u)=f(v)$ iff $\pi_{T''}(u)=\pi_{T''}(v)$.
We claim that $T'=T''\cong S$.
$T'$ must equal $T''$, since $f$ is injective.
Moreover, $\pi_{T''}(u)=\pi_{T''}(v)$ iff $\pi_S(u)=\pi_S(v)$.
Hence, $T''\cong S$.

(4)
If $S\cong T$, then   $\mathcal{Y}^{\al}_S\cong\mathcal{Y}^{\al}_{T}$
by
Proposition \ref{prop.structureY_S} (5).
If $\mathcal{Y}^{\al}_S\cong\mathcal{Y}^{\al}_{T}$,
then by applying (3) twice, we find that $S$ and $T$ are isomorphic to subsets of each other.
Hence, $S\cong T$.
\end{proof}

The next Corollary follows immediately from  Proposition \ref{prop.structureY_S} and Theorem \ref{prop.RKstructure}, thus, recovering Laflamme's Theorem \ref{thm.Laflammethms} (3).

\begin{cor}\label{cor.RKchain} $\lgl\mathcal{Y}^{\al}_{\beta}:\beta\le\al\rgl$, forms a strictly decreasing chain of nonprincipal rapid p-points in the Rudin-Keisler ordering, with $\mathcal{Y}^{\al}_0$ Rudin-Keisler maximal and $\mathcal{Y}^{\al}_{\al}$ Rudin-Keisler  minimal in the chain.
Moreover, this chain is maximal within the  ordering of nonprincipal ultrafilters Rudin-Keisler reducible  $\mathcal{U}_{\al}$.
\end{cor}

We will extend the previous corollary to the setting of Tukey reducibility in Theorem \ref{cor.1Tpred}.

\begin{thm}\label{thm.Tukeystructure}
Let $n<\om$ and $S\in\mathfrak{S}_{\al}(n)$.
Let $\beta\le\al$ be minimal such that 
$S_{\beta}$ embeds into $S$.
Then $\mathcal{Y}^{\al}_S\equiv_T\mathcal{Y}^{\al}_{\beta}$.
\end{thm}

\begin{proof}
Let $n<\om$ and $S\in\mathfrak{S}_{\al}(n)$.
Let $\beta\le\al$ be minimal such that $S$ contains an isomorphic copy  of $S_{\beta}$, call it $S'$.
Thus, $S'$ is a downward closed chain in $S$ with largest order type among all chains in $S$, namely o.t.$([\beta,\al])$. 
The projection map $\pi_{S'}:B_S\ra B_{S'}$
witnesses that $\mathcal{Y}^{\al}_{S'}\le_{RK}\mathcal{Y}^{\al}_S$.
Since $\mathcal{Y}^{\al}_{S'}$ is isomorphic to $\mathcal{Y}^{\al}_{\beta}$,
we have that $\mathcal{Y}^{\al}_{\beta}\le_{RK}\mathcal{Y}^{\al}_S$.
Hence, $\mathcal{Y}^{\al}_{\beta}\le_T\mathcal{Y}^{\al}_S$.

For the reverse inequality,
first note that for each $X\in\mathcal{C}_{\al}$,
from $\pi_{S_{\beta}}(\mathcal{R}_{\al}(0)|X)$
one can  reconstruct $\pi_S(\mathcal{R}_{\al}(n)|X)$,
 since $n\ge 0$ and $\beta$ is minimal such that there is a member $s\in S$ with $\dom(s)=[\beta,\al]$.
Thus,
for each $X\in\mathcal{C}_{\al}$,
define 
$g(\pi_{S_{\beta}}(\mathcal{R}_{\al}(0)|X))
=\pi_S(\mathcal{R}_{\al}(n)|X)$. 
Then $g$ maps a cofinal subset of $\mathcal{Y}^{\al}_{\beta}$ cofinally and monotonically into $\mathcal{Y}^{\al}_S$.
Therefore,
$\mathcal{Y}^{\al}_S\le_T\mathcal{Y}^{\al}_{\beta}$.
\end{proof}

Now we are ready to prove the main theorem of this section, the classification of all Rudin-Keisler types of ultrafilters Tukey reducible to $\mathcal{U}_{\al}$.  The following notion of ultrafilter of $\vec{\mathcal{W}}$-tree encompasses the notion of iterated Fubini products of ultrafilters.

\begin{defn}\label{def.vecW-trees}
Let $\hat{\mathcal{T}}$ be  a well-founded tree,
let $\mathcal{T}$ denote the set of maximal nodes in $\hat{\mathcal{T}}$, and suppose that each $t\in\hat{\mathcal{T}}\setminus\mathcal{T}$ has infinitely many immediate successors in $\hat{\mathcal{T}}$.
For each $t\in\hat{\mathcal{T}}\setminus\mathcal{T}$,
let $\mathcal{W}_t$ be an ultrafilter on the base set consisting of all immediate successors of $t$ in $\hat{\mathcal{T}}$.
Let $\vec{\mathcal{W}}$ denote $(\mathcal{W}_t:t\in\hat{\mathcal{T}}\setminus\mathcal{T})$.
Then a {\em $\vec{\mathcal{W}}$-tree}
is a tree $T\sse\hat{\mathcal{T}}$ such that
for each $t\in T\cap(\hat{\mathcal{T}}\setminus\mathcal{T})$,
the collection of immediate successors of $t$ in $T$ is a member of the ultrafilter $\mathcal{W}_t$.
\end{defn}

\begin{thm}\label{thm.TukeyU_al}
Suppose 
$\mathcal{V}$ is a nonprincipal ultrafilter and  $\mathcal{V}\le_T\mathcal{U}_{\al}$.
Then $\mathcal{V}$ is isomorphic to an ultrafilter of  $\vec{\mathcal{W}}$-trees, where
$\hat{\mathcal{T}}\setminus\mathcal{T}$ is a well-founded tree, $\vec{\mathcal{W}}=(\mathcal{W}_t:t\in\hat{\mathcal{T}}\setminus\mathcal{T})$, 
  and each $\mathcal{W}_t$ is exactly $\mathcal{Y}^{\al}_S$ for some $n<\om$ and $S\in\mathfrak{S}_{\al}(n)$.
\end{thm}

\begin{proof}
The proof  is so similar to the proof of Theorem 5.10
in \cite{Dobrinen/Todorcevic11} that we only give a sketch of the proof, providing the few changes here.
By Proposition \ref{prop.W=frontuf},
 there is a front $\mathcal{F}$ on $\mathcal{C}_{\al}$ and  a function $f:\mathcal{F}\ra\bN$
 such that
$\mathcal{V}=f(\mathcal{U}_{\al}\re\mathcal{F})$.
 By Theorem \ref{thm.PRR_al} and the fact that $\mathcal{U}_{\al}$ is canonical for fronts,
 there is a $C\in\mathcal{C}_{\al}$ such that 
the equivalence relation induced by $f$ on $\mathcal{F}|C$ is canonical. 
Let $\vp$ denote the function from Theorem \ref{thm.PRR_al} which canonizes $f$.
If $\mathcal{F}=\{\emptyset\}$, then $\mathcal{V}$ is a principal ultrafilter, so we may assume that $\mathcal{F}\ne\{\emptyset\}$.

Let  $\mathcal{T}=\{\vp(a):a\in\mathcal{F}|C\}$.
Define $\mathcal{W}$ to be the filter on base set $\mathcal{T}$ generated by the sets
$\{\vp(a):a\in \mathcal{F}|X\}$,  $X\in\mathcal{C}_{\al}|C$.
For $X\in\mathcal{C}_{\al}|C$, 
let $\mathcal{T}|X$ denote
$\{\vp(a):a\in\mathcal{F}|X\}$.
By arguments in \cite{Dobrinen/Todorcevic11},
$\mathcal{W}$ is an ultrafilter which is isomorphic to $\mathcal{V}$.
Let $\hat{\mathcal{T}}$ denote the collection of all initial segments of elements of $\mathcal{T}$.
Precisely,
let $\hat{\mathcal{T}}$ be the collection of all $\vp(a)\cap r_i^{\al}(a)$ such that $a\in\mathcal{F}|C$, $i\le |a|$, and if $0<i<|a|$ then $S_{r^{\al}_i(a)}\ne \{\emptyset\}$.
 $\hat{\mathcal{T}}$ forms  a tree under the end-extension ordering.
Recall 
from the proof of Theorem \ref{thm.PRR_al}
that for $t\in\hat{\mathcal{T}}\setminus\mathcal{T}$,
for all $a,b\in\mathcal{F}$,
if $j<|a|$ is maximal such that $\vp(r^{\al}_j(a))=t$ and $k$ is maximal such that $\vp(r^{\al}_{k}(b))=t$,
then $S_{r^{\al}_j(a)}$ is isomorphic to $S_{r^{\al}_{k}(b)}$, and these are both not $\{\emptyset\}$.

For $t\in\hat{\mathcal{T}}\setminus\mathcal{T}$, 
define $\mathcal{W}_t$ to be the filter
generated by the sets
$\{\vp_{r^{\al}_j(a)}(u): u\in\mathcal{R}_{\al}(j)|X/a\}$,
for all $a\in\mathcal{F}|C$ such that $t\sqsubseteq \vp(a)$ and
$j<|a|$ maximal such that $\vp(r_j^{\al}(a))=t$, and all $X\in\mathcal{C}_{\al}|C$.
The base set for $\mathcal{W}_t$ is
$\{\pi_{S_{r^{\al}_j(a)}}(u):u\in\mathcal{R}_{\al}(j)|C\}$.
By arguments in \cite{Dobrinen/Todorcevic11}, it follows that 
for each $t\in\hat{\mathcal{T}}\setminus\mathcal{T}$,
$\mathcal{W}_{t}$ is an ultrafilter;
moreover,
for any  $a\in\mathcal{F}$ and $j<|a|$ maximal such that $\vp(r_j^{\al}(a))=t$, 
$\mathcal{W}_{t}$
 is generated by the collection of  $\{\vp_{r_j(a)}(u):u\in\mathcal{R}_{\al}(j)|X\}, X\in\mathcal{C}_{\al}|C$.
This follows from the fact that $\mathcal{U}_{\al}$ is Ramsey for $\mathcal{R}_{\al}$.

\begin{claim}\label{claim.thmF,T3}
Let $t\in\hat{\mathcal{T}}\setminus\mathcal{T}$.
Then $\mathcal{W}_{t}$ equals  $\mathcal{Y}_S^{\al}$ for some $n<\om$ and $S\in\mathfrak{S}_{\al}(n)$.
\end{claim}

\begin{proof}
Fix $a\in\mathcal{F}|C$ and $j<|a|$  with $j$ maximal such that $\vp(r_j^{\al}(a))=t$.
Let $S$ denote $S_{r^{\al}_j(a)}$.
For each $X\in\mathcal{C}_{\al}|C$,
$\{\vp_{r_j^{\al}(a)}(u):u\in\mathcal{R}_{\al}(j)|X\}= \pi_S(\mathcal{R}_{\al}(j)|X)\in\mathcal{Y}_S^{\al}$.
Since $\mathcal{W}_t$ is a nonprincipal ultrafilter,  $\mathcal{W}_t$ must equal  $\mathcal{Y}_S^{\al}$, by Fact \ref{fact.ufequal}.
\end{proof}

Thus,
$\mathcal{W}$ is  the  ultrafilter of  $\vec{\mathcal{W}}$-trees, where  $\vec{\mathcal{W}}=(\mathcal{W}_t:t\in\hat{\mathcal{T}}\setminus\mathcal{T})$.
This follows from the fact that for each $\vec{\mathcal{W}}$-tree $\hat{T}\sse\hat{\mathcal{T}}$,
$[\hat{T}]$ is a member of $\mathcal{W}$.
Thus,
$\mathcal{V}$ is isomorphic to the ultrafilter $\mathcal{W}$ on base set $\mathcal{T}$ generated by the $\vec{\mathcal{W}}$-trees.
\end{proof}

By Corollary \ref{cor.RKchain} and Theorems  \ref{thm.Tukeystructure} and \ref{thm.TukeyU_al}, we obtain the analogue of Laflamme's result for the Rudin-Keisler ordering now in the context of Tukey types.

\begin{thm}\label{cor.1Tpred}
Let $\al<\om_1$ and suppose $\mathcal{V}$ is a nonprincipal ultrafilter such that 
 $\mathcal{V}\le_T\mathcal{U}_{\al}$.
Then  there is a $\beta\le\al$ such that 
$\mathcal{V}\equiv_T\mathcal{Y}^{\al}_{\beta}$.
Thus, the collection of the Tukey types of all nonprincipal ultrafilters Tukey reducible to $\mathcal{U}_{\al}$ forms a decreasing chain of rapid p-points of order type $(\al+1)^*$. 
\end{thm}

\begin{proof}
Let $\mathcal{V}$ be a  nonprincipal ultrafilter such that  $\mathcal{V}\le_T\mathcal{U}_{\al}$.
Theorem \ref{thm.TukeyU_al} implies that $\mathcal{V}$ is isomorphic, and hence Tukey equivalent, to the ultrafilter on $\mathcal{T}$ generated by the  $\vec{\mathcal{W}}$-trees, where
 for each $t\in\hat{\mathcal{T}}\setminus\mathcal{T}$, the ultrafilter $\mathcal{W}_t$ is $\mathcal{Y}^{\al}_{S_t}$ for some $n<\om$ and some  $S_t\in\mathfrak{S}_{\al}(n)$.
By Theorem \ref{thm.Tukeystructure},
for each $t$, there is a $\beta_t\le\al$ such that  $\mathcal{Y}^{\al}_{S_t}\equiv_T \mathcal{Y}^{\al}_{\beta_t}$.
It follows that $\mathcal{V}$ is Tukey equivalent to the ultrafilter of $\lgl \mathcal{Y}^{\al}_{\beta_s}:s\in
\hat{\mathcal{S}}\setminus\mathcal{S}\rgl$-trees.
By induction on the lexicographical rank of $\mathcal{T}$, one concludes that 
the ultrafilter of $\lgl \mathcal{Y}^{\al}_{\beta_s}:s\in
\hat{\mathcal{S}}\setminus\mathcal{S}\rgl$-trees is
 Tukey equivalent to $\mathcal{Y}^{\al}_{\beta}$, where $\beta=\min\{\beta_s:
 s\in
\hat{\mathcal{S}}\setminus\mathcal{S}\}$.

Now suppose that $\gamma<\beta\le\al$ and suppose toward a contradiction that $\mathcal{Y}^{\al}_{\beta}\le_T\mathcal{Y}^{\al}_{\gamma}$.
Then there is a continuous monotone cofinal map $h:\mathcal{Y}^{\al}_{\gamma}\ra\mathcal{Y}^{\al}_{\beta}$, since $\mathcal{Y}^{\al}_{\gamma}$ and $\mathcal{Y}^{\al}_{\beta}$ are p-points.
Since $\mathcal{Y}^{\al}_{\gamma}\le_{RK}\mathcal{U}_{\al}$, 
let $g:\bT_{\al}\ra B_{S_{\gamma}}$ be such that $g(\mathcal{U}_{\al})=\mathcal{Y}^{\al}_{\gamma}$.
Then $h\circ g:\mathcal{U}_{\al}\ra\mathcal{Y}^{\al}_{\beta}$ is a continuous monotone cofinal map. 
By Proposition \ref{prop.W=frontuf},
there is a front $\mathcal{F}$ and a function $f:\mathcal{F}\ra B_{S_{\beta}}$ 
such that $\mathcal{Y}^{\al}_{\beta}= f(\lgl\mathcal{U}_{\al}\re\mathcal{F}\rgl)$.
By Theorem \ref{thm.PRR_al},
there is a $C\in\mathcal{C}_{\al}$
 such that $f\re\mathcal{F}|C$ is canonical, witnessed by the inner function $\vp$.
Noting that for each $X\in\mathcal{C}_{\al}|C$, $f(\mathcal{F}|X)\sse h\circ g(X)$ and $g(X)\sse B_{\gamma}$, we see that $\vp$ cannot distinguish between members $a,b\in\mathcal{F}$ for which $\pi_{S_{\gamma}}(a)=\pi_{S_{\gamma}}(b)$; contradiction to $f(\lgl \mathcal{U}_{\al}\re\mathcal{F}\rgl)
=\mathcal{Y}^{\al}_{\beta}$.

By Corollary \ref{cor.RKchain}, the $\mathcal{Y}^{\al}_{\beta}$, $\beta\le\al$, form a maximal chain in the Tukey ordering of ultrafilters Tukey reducible to $\mathcal{U}_{\al}$.

The second half follows from Theorems \ref{thm.Tukeystructure} and \ref{thm.TukeyU_al}.
\end{proof}

\begin{rem}
It follows from Theorem \ref{cor.1Tpred}  that the  Tukey equivalence class of $\mathcal{Y}^{\al}_{\beta}$ consists exactly of those ultrafilters which are isomorphic to some ultrafilter of $\vec{\mathcal{W}}$-trees, where for each $t\in\hat{\mathcal{T}}\setminus\mathcal{T}$, $\mathcal{W}_t\cong \mathcal{Y}^{\al}_{S_t}$ for some $S_t$ satisfying the following: for each $t\in \hat{\mathcal{T}}\setminus\mathcal{T}$, if $S_{\gamma}$ embeds into $S_t$, then $\gamma\ge\beta$; and for at least one $t\in \hat{\mathcal{T}}\setminus\mathcal{T}$, $S_{\beta}$ embeds into $S_t$.
\end{rem}

We conclude by pointing out some of the interesting structures that occur within the
 Tukey types of the ultrafilters $\mathcal{Y}^{\al}_{\beta}$.

\begin{examples}[Rudin-Keisler Structures within  Tukey Types]\label{ex.RK_T}
The Tukey type of $\mathcal{U}_{\al}$ contains all isomorphism types of countable iterations of Fubini products of $\mathcal{U}_{\al}$.
Hence, the Tukey type of $\mathcal{U}_{\al}$ contains a Rudin-Keisler strictly increasing chain of order type $\om_1$.
The Tukey type of $\mathcal{U}_{\al}$ contains
a rich array of  ultrafilters which are Rudin-Keisler incomparable.
For example,
it follows by arguments using the Abstract Ellentuck Theorem that $\mathcal{U}_{\al}\cdot\mathcal{U}_{\al}$ and $\mathcal{Y}^{\al}_{\bS_{\al}(n)}$ are Rudin-Keisler incomparable, for each $n\ge 2$. 
Furthermore, for $k<l<m<n$, $\mathcal{Y}^{\al}_{\bS_{\al}(k)}\cdot\mathcal{Y}^{\al}_{\bS_{\al}(n)}$ and $\mathcal{Y}^{\al}_{\bS_{\al}(l)}\cdot\mathcal{Y}^{\al}_{\bS_{\al}(m)}$ are Tukey equivalent to $\mathcal{U}_{\al}$ but are Rudin-Keisler incomparable with each other.
More examples can be made similarly, using iterated Fubini products.

For each $1\le\beta\le\al$, the Tukey class of $\mathcal{Y}^{\al}_{\beta}$ contains many Rudin-Keisler incomparable ultrafilters.
For instance, let  $S,T\in\bigcup_{n<\om}\mathfrak{S}_{\al}(n)$ be such that $S_{\beta}$ embeds into
both $S$ and $T$; for $\gamma<\beta$, $S_{\gamma}$ does not embed into $S$ and $S_{\gamma}$ does not embed into $T$; and neither of $S$ and $T$ embeds into the other.
Then $\mathcal{Y}^{\al}_S\equiv_T\mathcal{Y}^{\al}_T\equiv_T \mathcal{Y}^{\al}_{\beta}$.
However, $\mathcal{Y}^{\al}_S$ and $\mathcal{Y}^{\al}_T$ are Rudin-Keisler incomparable.

The collection of all ultrafilters Tukey reducible to $\mathcal{U}_{\al}$ includes the 
following Rudin-Keisler strictly decreasing chain of rapid p-points of order type $\al+1$:
$\mathcal{Y}^{\al}_{0}>_{RK}\mathcal{Y}^{\al}_1>_{RK}\dots>_{RK}\mathcal{Y}_{\al}^{\al}$.
Since each of  the $\mathcal{Y}^{\al}_{\beta}$, $\beta\le\al$, is a p-point,  none of the ultrafilters in this chain is a Fubini product of any other ultrafilters.
Moreover, it follows from Theorem 
\ref{thm.TukeyU_al} that this chain is Rudin-Keisler-maximal within the collection of ultrafilters Tukey reducible to $\mathcal{U}_{\al}$.
This chain is also Tukey-maximal decreasing within this collection, by Theorem \ref{cor.1Tpred}.

For any $\beta\le\al$,
the collection of all ultrafilters Tukey reducible to $\mathcal{Y}^{\al}_{\beta}$ includes 
the Rudin-Keisler decreasing chain $\mathcal{Y}^{\al}_{\beta}>_{RK}\dots>_{RK}\mathcal{Y}^{\al}_{\al}$.
In addition, it contains
many ultrafilters which are Tukey incomparable, and hence Rudin-Keisler incomparable.
Since all $\mathcal{Y}^{\al}_{\gamma}$ are rapid p-points, it follows from the results in this section and
Corollary 21 of \cite{Dobrinen/Todorcevic10} that 
for any $\gamma<\delta<\varepsilon<\zeta\le \beta$,
$\mathcal{Y}^{\al}_{\gamma}\cdot\mathcal{Y}^{\al}_{\zeta}$ and $\mathcal{Y}^{\al}_{\delta}\cdot\mathcal{Y}^{\al}_{\varepsilon}$ are both Tukey reducible to $\mathcal{Y}^{\al}_{\beta}$, and are Tukey incomparable with each other.
More general examples may be constructed by the interested reader using iterated Fubini products of appropriate subsets of ultrafilters from among $\mathcal{Y}^{\al}_{\gamma}$, $\gamma\le\beta$.
\end{examples}

\bibliographystyle{plain}
\bibliography{referencesR_1}

\end{document}